\documentclass[12pt, a4paper]{amsart}
\usepackage{mathtools}
\mathtoolsset{showonlyrefs,showmanualtags}
\usepackage[pagebackref]{hyperref} 
\hypersetup{
	colorlinks=true,       
	linkcolor=blue,          
	citecolor=magenta,        
	filecolor=magenta,      
	urlcolor=cyan           
}

\usepackage[normalem]{ulem}

\allowdisplaybreaks

\usepackage{amsmath,amsthm,amscd,amssymb,amsfonts, amsbsy,mathrsfs}
\usepackage{latexsym}
\usepackage{exscale}

\usepackage[utf8]{inputenc}

\calclayout

\numberwithin{equation}{section}

\theoremstyle{plain}
\newtheorem{Theorem}[equation]{Theorem}

\newtheorem{Corollary}[equation]{Corollary}
\newtheorem{Lemma}[equation]{Lemma}

\theoremstyle{definition}
\newtheorem{Remark}[equation]{Remark}

\theoremstyle{definition}
\newtheorem{Definition}[equation]{Definition}
\def\Xint#1{\mathchoice
{\XXint\displaystyle\textstyle{#1}}%
{\XXint\textstyle\scriptstyle{#1}}%
{\XXint\scriptstyle\scriptscriptstyle{#1}}%
{\XXint\scriptscriptstyle\scriptscriptstyle{#1}}%
\!\int}
\def\XXint#1#2#3{{\setbox0=\hbox{$#1{#2#3}{\int}$}
\vcenter{\hbox{$#2#3$}}\kern-.5\wd0}}

\def\dashint{\Xint-}

\newcommand{\R}{\mathbb{R}}
\newcommand{\C}{\mathbb{C}}

\newcommand{\N}{\mathbb{N}}
\newcommand{\Z}{\mathbb{Z}}
\newcommand{\E}{\mathbb{E}}

\newcommand{\calS}{\mathcal{S}}

\newcommand{\bla}{\big \langle}
\newcommand{\bra}{\big \rangle}

\numberwithin{equation}{section}

\newcommand{\esssup}[0]{\operatornamewithlimits{ess\,sup}}

\newcommand{\supp}[0]{\operatorname{supp}}
\newcommand{\loc}[0]{\operatorname{loc}}

\newcommand{\ch}[0]{\operatorname{ch}}

\newcommand{\calD}[0]{\mathcal{D}}

\newcommand{\re}{\mathbb{R}}

\DeclareMathOperator*{\essinf}{ess\,inf}

\begin{document}

\title[End-point extrapolation and applications]{End-point estimates, extrapolation for multilinear Muckenhoupt classes, and applications}

\author[K. Li]{Kangwei Li}
\address{Kangwei Li, \textit{Current address:} Center for Applied Mathematics, Tianjin University, Weijin Road 92, 300072 Tianjin, China}
\address{and}
\address{BCAM, Basque Center for Applied Mathematics
	\\
	Mazarredo 14
	\\
	E-48009 Bilbao, Spain}
\email{kli@tju.edu.cn}

\author[J.M. Martell]{Jos\'e Mar{\'\i}a Martell}
\address{Jos\'e Mar{\'\i}a Martell
	\\
	Instituto de Ciencias Matem\'aticas CSIC-UAM-UC3M-UCM
	\\
	Consejo Superior de Investigaciones Cient\'ificas
	\\
	C/ Nicol\'as Cabrera, 13-15
	\\
	E-28049 Madrid, Spain}

\email{chema.martell@icmat.es}

\author[H. Martikainen]{Henri Martikainen}

\address{Henri Martikainen
\\
Department of Mathematics and Statistics
\\
University of Helsinki
\\
P.O.B. 68, FI-00014 University of Helsinki, Finland}
\email{henri.martikainen@helsinki.fi}

\author[S. Ombrosi]{Sheldy Ombrosi}
\address{Sheldy Ombrosi
	\\
	Department of Mathematics
	\\
	Universidad Nacional del Sur
	\\
	Bah\'ia Blanca, Argentina}
\email{sombrosi@uns.edu.ar}

\author[E. Vuorinen]{Emil Vuorinen}
\address{Emil Vuorinen
	\\
	Centre for Mathematical Sciences
	\\
	University of Lund
	\\
	P.O.B. 118, 22100 Lund, Sweden}
\email{j.e.vuorin@gmail.com}
\date{February 12, 2019. \textit{Revised:} \today}



\begin{abstract}
In this paper we present the results announced in the recent work by the first, second, and fourth authors of the current paper concerning Rubio de Francia extrapolation for the so-called multilinear Muckenhoupt classes. Here we consider the situations where some of the exponents of the Lebesgue spaces appearing in the hypotheses and/or in the conclusion can be possibly infinity. The scheme we follow is similar, but, in doing so, we need to develop a one-variable end-point off-diagonal extrapolation result. This complements the corresponding ``finite'' case obtained by Duoandikoetxea,  which was one of the main tools in the aforementioned paper. The second goal of this paper is to present some applications.   For example, we obtain the full range of mixed-norm estimates for tensor products of bilinear Calder\'on-Zygmund operators with a proof based on extrapolation and on some estimates with weights in some mixed-norm classes. The same occurs with the multilinear Calder\'on-Zygmund operators, the bilinear Hilbert transform, and the corresponding commutators with BMO functions. Extrapolation along with the already established weighted norm inequalities easily give scalar and vector-valued inequalities with multilinear weights and these include the end-point cases.
\end{abstract}

\maketitle

\section{Introduction}
Recently, the first, second, and fourth authors of the present paper solved in \cite{LMO} a long standing problem about the extrapolation theorem for multilinear  Muckenhoupt classes of weights. A particular and simplified version of the general result established there is the following: suppose that a multivariable operator $T$  satisfies
\[
\|T(f_1,\dots,f_m) w\|_{L^p}
\lesssim
\|f_1 w_1\|_{L^{p_1}}\dots \|f_m w_m\|_{L^{p_m}}
\]
for some fixed $1\le p_1,\dots, p_m<\infty$, and for all $\vec{w}=(w_1,\dots, w_m)\in A_{\vec{p}}$ (see Section \ref{sec:multi}), where $\frac1p=\frac1{p_1}+\dots+\frac1{p_m}$ and $w=\prod_{i=1}^m w_i$. Then, one gets the same kind of estimates for all $1< p_1,\dots,  p_m<\infty$, and for all $\vec{w}=(w_1,\dots,w_m)\in A_{\vec{p}}$. The one-variable case (i.e., $m=1$) is the well-known Rubio de Francia extrapolation theorem \cite{RdF} of which  one can find a great deal of extensions and refinements which are adapted to various settings and situations  (see \cite{CMP}).

The first result in the multivariable case \cite{GM}, due to the second author of this paper and Grafakos, was obtained under the assumption that  every individual weight is in the corresponding Muckenhoupt class. This collection of weights  is however too  small and the paper \cite{LOPTT}, coauthored by the fourth author of this paper together with Lerner, P\'erez, Torres, and Trujillo-Gonz\'alez, introduced and studied the classes $A_{\vec{p}}$. In these, rather than looking at the weights in the $m$-tuple individually, it is assumed that they collectively satisfy a Muckenhoupt-type condition. These are in turn the natural classes for the multilinear Calder\'on-Zygmund operators in the same way that the classical $A_p$ classes are the natural ones for the linear Calder\'on-Zygmund operators.

Since the appearance of \cite{LOPTT}, an associated Rubio de Francia extrapolation theory for the classes $A_{\vec{p}}$ has been sought until recently when the aforementioned paper \cite{LMO} solved that problem, and even more, obtained extrapolation results for the more general classes $A_{\vec{p},\vec{r}}$ (see Section \ref{sec:multi}). That paper, however, considered only the cases in which all the exponents are finite and it was announced that one could consider situations in which some of the exponents $p_i$'s and/or the $q_i$'s are infinity. While these end-point scenarios are not natural in the one-variable extrapolation theory, it turns out that for some relevant multilinear applications it is of interest to extend the extrapolation theory so that it includes the end-point cases. The main goal of this paper is to present in a rigorous way these end-point extrapolation results originally announced in \cite{LMO}. Since the statement of our main extrapolation result for the classes $A_{\vec{p},\vec{r}}$ (see Theorem \ref{theor:extrapol-general}) requires some notation, we postpone that until Section \ref{sec:multi}. However, we would like to single out the particular case of the original classes $A_{\vec{p}}$, which is of particular interest and also gives an-easy-to-digest illustration of the general case.

\begin{Theorem}\label{theor:extrapol-general:CZO}
	Let $m\ge 2$ and $T$ be an $m$-variable operator. Given $\vec p=(p_1,\dots, p_m)$ with $1\le p_1,\dots, p_m\le \infty$, let $\frac1p=\frac1{p_1}+\dots+\frac1{p_m}\in[0,m]$. Assume that given any $\vec w=(w_1,\dots, w_m) \in A_{\vec p}$ the inequality
	\begin{equation}\label{extrapol:H*}
	\|T(f_1,\dots, f_m)w\|_{L^{p}} \lesssim  \prod_{i=1}^m\|f_iw_i\|_{L^{p_i} }
	\end{equation}
	holds for every $(f_1,\dots, f_m)\in \mathbb{F}$ (a fixed collection of $m$-tuples of measurable functions), where $w:=\prod_{i=1}^m w_i $. Then for all exponents $\vec q=(q_1,\dots,q_m)$, with $1< q_1,\dots, q_m\le \infty$ and $\frac1q=\frac1{q_1}+\dots+\frac1{q_m}>0$, and for all weights $\vec v=(v_1,\dots, v_m) \in A_{\vec q}$ the inequality
	\begin{equation}\label{extrapol:C*}
	\|T(f_1,\dots, f_m) v\|_{L^{q} } \lesssim \prod_{i=1}^m \|f_iv_i\|_{L^{q_i} }
	\end{equation}
	holds for every $(f_1,\dots,f_m)\in \mathbb F$, where $v:=\prod_{i=1}^m v_i $.
	
	Moreover, for the same family of exponents and
	weights, and for all exponents $\vec{s}=(s_1,\dots, s_m)$ with $1< s_1,\dots, s_m\le \infty$ and $\frac1s=\frac1{s_1}+\dots+\frac1{s_m}>0$ the inequality
	\begin{equation}\label{extrapol:vv*}
	\Big\|\big\{T(f_1^j,\dots, f_m^j) v\big\}_j\Big\|_{L^{q}_{\ell^s}}
	\lesssim
	\prod_{i=1}^m\Big\|\big\{f_i^j v_i\big\}_j\Big\|_{L^{q_i}_{\ell^{s_i}} }
	\end{equation}
	holds for all sequences $\{(f_1^j, \dots, f_m^j)\}_j\subset\mathbb{F}$.
\end{Theorem}

As mentioned, this result is a particular case of Theorem \ref{theor:extrapol-general} and extends \cite[Corollary 1.5]{LMO} where all the $p_i$'s, $q_i$'s, $s_i$'s and $p$ are assumed to be finite. Note that Theorem \ref{theor:extrapol-general:CZO} allows us to extrapolate starting with some (or even all) $p_i$'s being infinity. Not only that, even if we start with all the $p_i$'s being finite, we can derive estimates with some (but not all) of the $q_i$'s being infinity.

The statement of Theorem \ref{theor:extrapol-general} and its proof are given in Section \ref{sec:multi}. For the latter we follow the blueprint established in \cite{LMO} with very minor changes. However, a key ingredient in the proof is Theorem \ref{thm:offdiagonal}, which is a one-variable off-diagonal result of independent interest. This is indeed an extension of \cite[Theorem 5.1]{Duo} (and also of \cite{HMS}) allowing us to start with or to obtain end-point estimates. Our proof is, however, slightly different and follows the scheme used thoroughly in \cite{CMP}.

Our second goal of this paper is to present some applications of our extrapolation result. In Section \ref{sec:apps} we briefly consider some of the examples already treated in \cite{LMO} and these include multilinear Calder\'on-Zygmund operators, multilinear sparse forms, bilinear rough singular integral operators, the bilinear Hilbert transform and their commutators with BMO functions. In Section \ref{section:tensor}, we present a new application which  concerns mixed-norm estimates  (a topical subject, see e.g. \cite{BM1, BM2} and \cite{DO}). Here we work with tensor products $T_n \otimes T_m$, where $T_n$ and $T_m$ are \emph{bilinear} Calder\'on-Zygmund operators in $\R^n$ and $\R^m$. The tensor product is initially defined via
$$
(T_n \otimes T_m)(f_1 \otimes f_2, g_1 \otimes g_2)(x) := T_n(f_1, g_1)(x_1)T_m(f_2, g_2)(x_2),
$$
where $f_1, g_1 \colon \R^n \to \C$, $f_2, g_2 \colon \R^m \to \C$ and $x = (x_1, x_2) \in \R^{n+m}$.
Notice that if $T_n$ and $T_m$ are linear, then
$T_n \otimes T_m = T_n^1T_m^2$, where $T_n^1 f(x) = T_n(f(\cdot, x_2))(x_1)$ and $T_m^2 f(x) = T_m(f(x_1,\cdot))(x_2)$. One can then easily obtain the desired estimates by using iterative arguments and Fubini's theorem, and therefore, in the linear case, bi-parameter singular integrals are only interesting if they are not of tensor product type. Unlike the linear case, the theory of tensor products of bilinear operators is already non-trivial. Indeed, Journ\'e in \cite{Jo3} showed that a tensor product of general bilinear operators, both bounded from $L^{\infty} \times L^2$ to $L^2$, needs not be bounded from $L^{\infty} \times L^2$ to $L^2$. On the other hand, he obtained positive results for tensor products of  some multilinear singular integral forms from Christ-Journ\'e \cite{CJ}.

As an application of our extrapolation results we are able to obtain (see Corollary \ref{cor:TensorEndpoint}) that
$$
\| T_n \otimes T_m (f_1,f_2) \|_{L^{p}(\R^n; L^{q}(\R^m))}
\lesssim \|f_1 \|_{L^{p_1}(\R^n; L^{q_1}(\R^m))}
\| f_2 \|_{L^{p_2}(\R^n; L^{q_2}(\R^m))}
$$
whenever $p_1,p_2, q_1,q_2 \in (1, \infty]$ are such that
$\frac1p=\frac1{p_1}+\frac1{p_2}>0$ and $\frac1{q}=\frac1{q_1}+\frac1{q_2}>0$. This should be compared with \cite{LMV1} where the first, third, and last authors of the present paper considered general bilinear bi-parameter Calder\'on-Zygmund operators (with no tensor form). Some mixed-norm estimates were proved as quick corollaries from  the obtained weighted estimates (and operator-valued analysis). However, since the classes of weights considered were of product type $A_{p_1} \times A_{p_2}$ (in place of the corresponding bilinear class $A_{(p_1,p_2)})$, in the mixed-norm estimates one needed to assume that for the case $q < 1$ one has $p_1$, $p_2<\infty$. Here we are able to obtain more sophisticated weighted estimates and this restriction can be removed. It should be noticed that here we only work with tensor products of bilinear singular integrals, but we can show somewhat stronger weighted estimates compared to those  in \cite{LMV1}: we essentially use weights which are in $A_{(p_1,p_2)}(\R^n)$ and in $A_{p_1}(\R^m) \times A_{p_2}(\R^m)$ and, by extrapolation, this is enough for the mixed-norm inequalities in question.

To conclude with this introduction we would like to mention that some interesting related work has recently appeared while this manuscript was in preparation. Nieraeth in \cite{Nier} has obtained a result similar to ours with a proof which uses an independent different method. In a nutshell, Theorem \ref{theor:extrapol-general} ---which was announced in \cite{LMO} before \cite{Nier} was posted--- is proved by following \cite[Proof of Theorem 1.1]{LMO} with some appropriate changes both in the argument and in the notation. Having said that, the main tool that we need here is Theorem \ref{thm:offdiagonal}, which extends \cite[Theorem 5.1]{Duo} (and also \cite{HMS}) to allow for end-point estimates. We also note that our proof of Theorem \ref{thm:offdiagonal} is somewhat different, and of independent interest, than that in \cite{Duo} and follows the scheme thoroughly employed in \cite[Chapter 3]{CMP}. On the other hand Duoandikoetxea and Oruetxebarria in \cite{duo-orue} have obtained mixed-norm estimates of radial-angular type by developing an extrapolation theory for radial weights.

\subsection*{Acknowledgments}
K. Li was supported by Juan de la Cierva - Formaci\'on 2015 FJCI-2015-24547, by the Basque Government through the BERC
2018-2021 program and by Spanish Ministry of Economy and Competitiveness
MINECO through BCAM Severo Ochoa excellence accreditation SEV-2017-0718
and through project MTM2017-82160-C2-1-P funded by (AEI/FEDER, UE) and
acronym ``HAQMEC''.

J.M. Martell was supported by the Spanish Ministry of Economy and Competitiveness,
through the ``Severo Ochoa Programme for Centres of Excellence in
R\&D'' (SEV-2015-0554) and the European
Research Council through the European Union's Seventh Framework
Programme (FP7/2007-2013)/ ERC agreement no. 615112 HAPDEGMT.

H. Martikainen was supported by the Academy of Finland through the grants 294840 and 306901, and by the three-year research grant 75160010 of the University of Helsinki.
He is a member of the Finnish Centre of Excellence in Analysis and Dynamics Research.

S. Ombrosi was supported by CONICET PIP 11220130100329CO, Argentina.

E. Vuorinen was supported by the Jenny and Antti Wihuri Foundation.

The authors would like to thank the anonymous referee for the useful suggestions that  helped us to improve the presentation of the paper.

\section{End-point extrapolation for multilinear Muckenhoupt classes}\label{sec:multi}

Our main result in this section contains the extension to the end-point cases announced in our previous paper  \cite{LMO}.  Before we state that result, let us recall some notations and make some conventions. Given a  cube $Q$, its side-length will be denoted by $\ell(Q)$ and for  any $\lambda>0$ we let  $\lambda Q$ be
the cube concentric with $Q$ whose side-length is $\lambda \ell
(Q)$. Let $\mu$ be a doubling measure on $\re^n$, that is, $\mu$ is a non-negative Borel regular measure such that $\mu(2Q)\le C_\mu \mu(Q)<\infty$ for every cube $Q\subset \re^n$. Given a Borel set $E\subset\re^n$ with $0<\mu(E)<\infty$ we use the notation
\[
\dashint_E f d\mu=\frac1{\mu(E)}\int_E fd\mu.
\]

This section is devoted to applying off-diagonal extrapolation to get the multivariable extrapolation in the end-point cases.

Hereafter, $m\ge 2$.  Given  $\vec p=(p_1,\dots, p_m)$ with $1\le
p_1,\dots, p_m\le \infty$ and $\vec{r}=(r_1,\dots,r_{m+1})$ with $1\le r_1,\dots,r_{m+1}<\infty$,
we say that $\vec{r}\preceq_\star \vec{p}$ whenever
\[
r_i\le p_i,
\quad
i=1,\dots, m;
\quad\mbox{and}
\quad r_{m+1}'\ge p,
\quad
\mbox{where}
\quad
\frac1p:=\frac1{p_1}+\dots+\frac1{p_{m}}
.
\]
This should be compared with the notation $\vec{r}\preceq\vec{p}$ from \cite{LMO} where one adds the restriction $r_{m+1}'> p$. Analogously, we say that $\vec{r}\prec \vec{p}$ if $\vec{r}\preceq_\star \vec{p}$ and moreover $r_i<p_i$ for every $i=1,\dots, m$ and $r_{m+1}'> p$.
Notice that the fact that $\vec{r}\preceq_\star\vec{p}$ forces that $\sum_{i=1}^{m+1}\frac1{r_i}\ge 1$ and also $\frac1 p\le \sum_{i=1}^m \frac1{r_i}$. Hence, if $\sum_{i=1}^m \frac1{r_i}>1$ then we allow $p$ to be smaller than one.

We can now introduce the classes of multilinear Muckenhoupt weights that we consider in the present paper. Given  $\vec p=(p_1,\dots, p_m)$ with $1\le p_1,\dots,p_m\le \infty$ and $\vec{r}=(r_1,\dots,r_{m+1})$ with $1\le r_1,\dots,r_{m+1}<\infty$ so that $\vec{r}\preceq_\star\vec{p}$ we say that   $\vec{w}=(w_1,\dots, w_m)\in A_{\vec p, \vec r}$, provided $0<w_i<\infty$ a.e. for every $i=1,\dots,m$ and
\[
[\vec{w}]_{A_{\vec p, \vec r}}
=\sup_Q\Big(\dashint_Q w^{\frac {r_{m+1}'p}{r_{m+1}'- p}} dx\Big)^{\frac 1{p}-\frac
	1{r_{m+1}'}}
\prod_{i=1}^m \Big(\dashint_Q w_i^{\frac{r_i p_i}{r_i-p_i}} d x\Big)^{\frac 1{r_i}-\frac
	1{p_i}}<\infty,
\]
where $w=\prod_{i=1}^m w_i$. When $p=r_{m+1}'$ the term corresponding to $w$ needs to be replaced by $\esssup_Q w$ and, analogously, when $p_i=r_i$, the term
corresponding to $w_i$ should be $\esssup_Q w_i^{-1}$. Also, if $p_i=\infty$ the term corresponding to $w_i$ becomes $\Big(\dashint_Q w_i^{-r_i} d x\Big)^{\frac
	1{r_i}}$. If $p=\infty$, one necessarily have $r_{m+1}=1$ and $p_1=\dots=p_m=\infty$, hence the term corresponding to $w$ must be $\esssup_Q w$ while the terms corresponding to $w_i$ become $\Big(\dashint_Q w_i^{-r_i} d x\Big)^{\frac
	1{r_i}}$. When $r_{m+1}=1$ and $p<\infty$ the term corresponding to $w$ needs to be replaced by $\Big(\dashint_Q w^p d x\Big)^{\frac
	1{p}}$. We remark that similar classes of weights, to be precise when $r_1=\cdots=r_m=r\ge 1$, were introduced earlier in \cite{CTW}.

Note that with the previous definition $\vec{w}\in A_{\vec p, (1,\dots,1)}$  means that
\[
[\vec{w}]_{A_{\vec p, (1,\dots,1)}}
=\sup_Q\Big(\dashint_Q w^{p} dx\Big)^{\frac 1{p}}
\prod_{i=1}^m \Big(\dashint_Q w_i^{-p_i'}d x\Big)^{\frac 1{p_i'}}<\infty,
\]
where $w=\prod_{i=1}^m w_i$ and with the appropriate changes when some $p_i=1$ or $p=\infty$. In the sequel we will just simply denote $A_{\vec p, (1,\dots,1)}$ by $A_{\vec p}$. We would like to observe our definition  of the classes ${A_{\vec p, \vec r}}$ is slightly different to that in \cite{LOPTT} or \cite{LMO} (notice that e.g. $[w_1^{p_1}, \cdots, w_m^{p_m}]_{A_{\vec p}}^{\frac 1p}$ with $A_{\vec p}$ defined in \cite{LOPTT} agrees with our $[\vec w]_{A_{\vec p}}$ when $p_i$'s are finite). This change is just cosmetic but it turns out to be useful for understanding the end-point estimates.

It is convenient to write the condition $A_{\vec p, \vec r}$ in a different form. With that goal in mind, given  $\vec p=(p_1,\dots, p_m)$ with $1\le p_1,\dots,p_m\le \infty$ and $\vec{r}=(r_1,\dots,r_{m+1})$
with $1\le r_1,\dots,r_{m+1}<\infty$ so that $\vec{r}\preceq_\star\vec{p}$ we set
\[
\frac1r:=\sum_{i=1}^{m+1} \frac1{r_i},
\qquad
\frac1{p_{m+1}}:=1-\frac1p,
\qquad\mbox{and}\qquad
\frac1{\delta_i}:=\frac1{r_i}-\frac1{p_i},
\quad i=1,\dots, m+1.
\]
Notice that as observed above we have that $0<r\le 1$ and formally $\frac1{p_{m+1}}=\frac1{p'}$ which  could be negative or zero if $p\le 1$. Note that in this way
\[
\sum_{i=1}^{m+1}\frac1{p_i}=1
\qquad\mbox{and}\qquad
\sum_{i=1}^{m+1}\frac1{\delta_i}=\frac1{r}-1=\frac{1-r}{r}.
\]
Also, $\vec{r}\preceq_\star\vec{p}$ means that $\delta_i^{-1}\ge 0$ for every $1\le i\le m+1$. On the other hand, $\vec{r}\prec\vec{p}$ means that $\delta_i^{-1}>0$ for every $1\le i\le m+1$.
Notice that with this notation $\vec w=(w_1, .., w_m)\in A_{\vec{p},\vec{r}}$  can be written as
\[
[\vec{w}]_{A_{\vec p, \vec r}}
=
\sup_Q\Big(\dashint_Q w^{  {\delta_{m+1}} }d x\Big)^{\frac 1{\delta_{m+1}}}
\prod_{i=1}^m \Big(\dashint_Q w_i^{- {\delta_i} }d x\Big)^{\frac1{\delta_i}}
<\infty,
\]
where $w=\prod_{i=1}^m w_i$.

Again we shall use the abstract formalism of extrapolation families. Hereafter $\mathcal{F}$ will denote a family of $(m+1)$-tuples $(f,f_1,\ldots,f_m)$ of non-negative
measurable functions.

Our main result is the following:

\begin{Theorem}\label{theor:extrapol-general}
	Let $ \mathcal F$ be a collection of $(m+1)$-tuples of non-negative functions. 	Consider a vector $\vec{r}=(r_1,\dots,r_{m+1})$, with $1\le r_1,\dots,r_{m+1}<\infty$, and
	exponents $\vec p=(p_1,\dots, p_m)$ with $1\le p_1,\dots,p_{m}\le\infty$, so that $\vec{r}\preceq_\star \vec{p}$. Assume that given any $\vec w=(w_1,\dots, w_m) \in A_{\vec p,\vec{r}}$ the inequality
	\begin{equation}\label{extrapol:H}
	\|fw\|_{L^{p}} \le \varphi([\vec w]_{A_{\vec p, \vec r}}) \prod_{i=1}^m\|f_iw_i\|_{L^{p_i} }
	\end{equation}
	holds for every $(f,f_1,\dots, f_m)\in \mathcal F$, where $\frac1p:=\frac1{p_1}+\dots+\frac1{p_m}$, $w:=\prod_{i=1}^m w_i $, and $\varphi\ge 0$ is a non-decreasing function.  	Then for all exponents $\vec q=(q_1,\dots,q_m)$, with $1<q_1,\dots, q_m\le \infty$, so that $\vec{r}\prec\vec{q}$, and for all weights $\vec v=(v_1,\dots, v_m) \in A_{\vec q,\vec{r}}$ the inequality
	\begin{equation}\label{extrapol:C}
	\|fv\|_{L^{q} } \le \widetilde{\varphi}([\vec v]_{A_{\vec q, \vec r}})\prod_{i=1}^m \|f_iv_i\|_{L^{q_i} }
	\end{equation}
	holds for every $(f,f_1,\dots,f_m)\in \mathcal F$, where $\frac1q:=\frac1{q_1}+\dots+\frac1{q_m}>0$, $v:=\prod_{i=1}^m v_i $, and $\widetilde{\varphi}\ge 0$ is a non-decreasing function.
	
	Moreover, for the same family of exponents and
	weights, and for all exponents $\vec{s}=(s_1,\dots, s_m)$ with  $1<
	s_1,\dots, s_m\le \infty$, so that $\vec{r}\prec\vec{s}$
	\begin{equation}\label{extrapol:vv}
	\Big\|\big\{f^jv\}_j\Big\|_{L^{q}_{\ell^s}}
	\le
	\widetilde{\varphi}([\vec v]_{A_{\vec q, \vec r}})
	\prod_{i=1}^m \Big\|\big\{f_i^j v_i\}_j\Big\|_{L^{q_i}_{\ell^{s_i}} }
	\end{equation}
	for all $\{(f^j, f_1^j, \dots, f_m^j)\}_j\subset\mathcal{F}$ and where $\frac1s:=\frac1{s_1}+\dots+\frac1{s_m}>0$.
\end{Theorem}
We would like to remark that this result was announced in \cite{LMO} and the main difference with \cite[Theorem 1.1]{LMO} is that here we allow the $p_i$'s and/or the $q_i$'s to take the value  infinity. Also, in the current result we can start with $p=r_{m+1}'$ (including the case $p=\infty$ if $r_{m+1}=1$) while in the conclusion we obtain $q<\infty$, since $\vec{r}\prec\vec{q}$. We also remark that if we start with $p_{i_0}=r_{i_0}$ for some given $i_0$ then in the conclusion we can relax $q_{i_0}>r_{i_0}$ to $q_{i_0}\ge r_{i_0}$, see \cite[Remark 1.8]{LMO}.

\subsection{End-point off-diagonal extrapolation theorem}
In this section we pre\-sent an end-point off-diagonal extrapolation result which is going to play an crucial role in the proof of  Theorem \ref{theor:extrapol-general}.
Next we give the basic properties of weights that we will
need below.  For proofs and further information,
see~\cite{duoandikoetxea01, grafakos08b}.   By a weight we mean a
measurable  function $v$ such that
$0<v<\infty$ $\mu$-a.e.   For $1<p<\infty$, we say that $v\in
A_p(\mu)$ if
\[
[v]_{A_p(\mu)} = \sup_Q \dashint_Q v\,d\mu
\left(\dashint_Q v^{1-p'}\,d\mu\right)^{p-1} < \infty, \]
where the supremum is taken over all cubes $Q\subset \re^n$. The quantity $[v]_{A_p(\mu)}$ is called
the $A_p(\mu)$ constant of $v$.  Notice that it follows at once from this
definition that if $v\in A_p(\mu)$, then $v^{1-p'}\in A_{p'}(\mu)$.  When $p=1$
we say that $v\in A_1(\mu)$ if
\[  [v]_{A_1(\mu)} = \sup_Q\left( \dashint_Q v\,d\mu\right) \esssup_{Q} v^{-1}<
\infty, \]
where the essential supremum is taken with respect to the underlying doubling measure $\mu$.
The $A_p(\mu)$ classes are properly nested:  for $1<p<q$, $A_1(\mu)\subsetneq
A_p(\mu) \subsetneq A_q(\mu)$.
We denote the union of all the $A_p(\mu)$ classes, $1\leq p<\infty$, by
$A_\infty(\mu)$.

Given $1\le p\le \infty$ and $0<r\le \infty$  we say that $v\in A_{p,r}(\mu)$ if
\[
[v]_{A_{p,r}(\mu)}
=
\sup_Q  \left(\dashint_Q v^{r}\,d\mu\right)^\frac1r
\left( \dashint_Q v^{-p'}\,d\mu\right)^{\frac{1}{p'}}<\infty, \]
where one has to replace the first term by $\esssup_Q v$ when $r=\infty$ and the second term by $\esssup_{Q} v^{-1}$ when $p=1$, and these essential suprema are taking with respect to $\mu$. The case $p=1$ and $r=\infty$ gives a trivial class of weights since $v\in A_{1,\infty}$ amounts to say that $v\approx 1$ $\mu$-a.e.. Assuming that we are not in that case we always have
\begin{equation}\label{gamma:rp}
\gamma_{p,r}:= \frac 1{r}+ \frac 1{p'}>0
\end{equation}
and one can easily see that $v\in A_{p,r}(\mu)$ if and only if $v^r\in A_{r\gamma_{p,r}}(\mu)$ if and only if $v^{-p'}\in A_{p'\gamma_{p,r}}$ with
\[
[v]_{A_{p,r}(\mu)}
=
[v^r]_{A_{r\gamma_{p,r}}(\mu)}^{\frac1r}
=
[v^{-p'}]_{A_{p'\gamma_{p,r}}(\mu)}^{\frac1{p'}}
\]
when $1<p\le\infty$ and $0<r<\infty$. If $p=1$ and $0<r<\infty$ then $\gamma_{p,r}:=\frac1r$ and $v\in A_{p,r}(\mu)$ if and only if $v^r\in A_{1}(\mu)$ with
$[v]_{A_{p,r}(\mu)}
=
[v^r]_{A_{1}(\mu)}^{\frac1r}$. Also, if $1<p\le\infty$  and $r=\infty$ then $\gamma_{p,r}:=\frac1{p'}$ and $v\in A_{p,r}(\mu)$ if and only if $v^{-p'}\in A_{1}(\mu)$ with
$[v]_{A_{p,r}(\mu)}
=
[v^{-p'}]_{A_{1}(\mu)}^{\frac1{p'}}$. We remark that the current definition of $A_{p,r}$ constant is slightly different with the one that was used in \cite{Duo} and \cite{LMO},  because we have to take care of the case $r=\infty$.

When $\mu$ is the Lebesgue measure we will simply write $A_p$, $A_{p,r}$, \dots. It is well-known that if $w\in A_\infty$ then $dw=w(x)dx$ is a doubling measure. Besides, since $0<w<\infty$ a.e., then the Lebesgue measure and  $w$ have the same null measure sets hence the essential suprema and infima with respect to the Lebesgue measure and $w$ agree.

We shall use the abstract formalism of extrapolation families.    Hereafter $\mathscr{F}$ will
denote a family of  pairs $(f,g)$ of non-negative
measurable functions.   This approach to extrapolation has the
advantage that, for instance, vector-valued inequalities are an
immediate consequence of our extrapolation results.  We will discuss
applying this formalism to prove norm inequalities for specific operators
below.  For complete discussion of this
approach to extrapolation in the linear setting, see~\cite{CMP}.

The main result of this section is the following off-diagonal Rubio de Francia extrapolation result which extends \cite[Theorem 5.1]{Duo} (and also \cite{HMS}) allowing us to start with or to obtain end-point results:

\begin{Theorem}\label{thm:offdiagonal}
Let $ \mathscr F$ be a collection of pairs of non-negative functions. Let $1\le p_0\le \infty$, $0<r_0\le\infty$, and $0<q_0\le \infty$.
Assume that for all $w\in A_{p_0, r_0}$ we have
\begin{equation}\label{off-extrapol:hyp}
\| f w\|_{L^{q_0}}\le \varphi([w]_{A_{p_0, r_0}})\|gw\|_{L^{p_0}},
\end{equation}
for all $(f,g)\in\mathscr{F}$ and where $\varphi\ge 0$ is a non-decreasing function.
Then for all $1<p\le\infty$, $0<r<\infty$, and $0<q\le \infty$ such that
\begin{equation}\label{off-extrapol:compat}
\frac 1q-\frac 1{q_0}= \frac 1r- \frac 1{r_0}= \frac 1p-\frac 1{p_0},
\end{equation}
and all $w\in A_{p, r}$ we have
\begin{equation}\label{off-extrapol:conc}
\| f w\|_{L^{q}}\le \widetilde{\varphi}([w]_{A_{p, r}})\|gw\|_{L^{p}},
\end{equation}
for all $(f,g)\in\mathscr{F}$ and where $\widetilde{\varphi}\ge 0$ is a non-decreasing function.
\end{Theorem}

We observe that this result is on the nature of best possible in the sense that one cannot expect to reach the end-points $p=1$ and/or $r=\infty$. The fact that one cannot obtain $p=1$ is well-known. 
To see that we cannot extrapolate to $r=\infty$ we let $p_0=q_0=r_0=2$ (again $\gamma_{r_0,p_0}=1>0$) and consider the pairs $(|Hf|,|f|)$ with $f\in L^\infty_c(\re)$ and where $H$ is the Hilbert transform.  Note that \eqref{off-extrapol:hyp} follows since the Hilbert transform is bounded on $L^2(w^2)$ for every $w^2\in A_2$. If we could extrapolate to
$r=\infty$,  picking $p=q=r=\infty$ we would obtain $\|H f w\|_{L^\infty}\lesssim \|f w\|_{L^\infty(w)}$ for every $w\in A_{\infty,\infty}$, that is, for any $w^{-1}\in A_1$. Taking in particular $w\equiv 1$  we would obtain that $Hf\in L^\infty(\re^n)$ for every $f\in L^\infty_c(\re)$ leading contradiction.

Even more we can see that we cannot extrapolate to $r=\infty$ with $0<q<\infty$. This requires some extra work. For any $f\in L^\infty_c(\re)$,  let $E_f$ be the set of Lebesgue points for the function $|Hf|^2\in L^1(\re)$ so that $|\re\setminus E_f|=0$. Define, for every $x\in\re$ and $0<\tau<\infty$, the non-negative function $g_{x,\tau}=\tau^{-\frac12}\mathbf{1}_{(x-\tau/2, x+\tau/2)}$ so that $\|g_{x,\tau}\|_{L^2(\re)}=1$ and consider the family
\[
\mathscr{F}=\big\{(|Hf|\,g_{x,\tau}, |f|):\ f\in L^\infty_c(\re),\ x\in E_f,\ 0<\tau<\infty \big\}.
\]
Let $q_0=1$ and $p_0=r_0=2$, and note that $\gamma_{r_0,p_0}=1>0$. For every $w\in A_{2,2}$ (that is, $w^2\in A_2$) one has for every pair $(F,G)=(|Hf|\,g_{x,\tau}, f)\in\mathscr{F}$
\begin{multline*}
\|F w\|_{L^1(\re)}
=
\big\||Hf|\,g_{x,\tau} w\big\|_{L^1(\re)}
\le
\|Hf\|_{L^2(w^2)}\|g_{x,\tau}\|_{L^2(\re)}
\\
=
\|Hf\|_{L^2(w^2)}
\le
C_w \|f\|_{L^2(w^2)}
=
C_w \|G\|_{L^2(w^2)}.
\end{multline*}
If we could extrapolate to $r=\infty$ then we would pick $p=\infty$, $q=2$, and $r=\infty$ so that \eqref{off-extrapol:compat} holds to obtain that
$\|F w\|_{L^2(\re)}\lesssim \|Gw\|_{L^\infty(\re)}$ for all $(F,G)\in \mathscr{F}$ and $w\in A_{\infty,\infty}$, that is, $w^{-1}\in A_1$. We could take again $w\equiv 1$ to see that for every $(F,G)=(|Hf|\,g_{x,\tau}, f)\in\mathscr{F}$, that is, for every $f\in L^\infty_c$, $x\in E_f$ and $0<\tau<\infty$
\[
\left(\dashint_{(x_0-\tau/2,x_0+\tau/2)} |Hf(y)|^2\right)^{\frac12}
=
\|F\|_{L^2(\re)}
\le
C_0
\|G\|_{L^\infty(\re)}
=
\|f\|_{L^\infty(\re)}.
\]
Since $C_0$ does not depend on $f$, $x$ and $\tau$ (note that $C$ is the same for all the pairs in the family $\mathscr{F}$), and $x\in E_f$ we could let $\tau\to 0^+$ and conclude that
\[
|Hf(x)|\le C_0 \|f\|_{L^\infty},
\qquad\forall\, x\in E_f.
\]
Consequently, we would obtain that $Hf\in L^\infty$ for every $f\in L^\infty_c$ which is again a contradiction.

\begin{proof}[Proof of Theorem \ref{thm:offdiagonal}]
The case $1\le p_0<\infty$, $0<r_0<\infty$, and $1<p<\infty$ was proved in \cite[Theorem 5.1]{Duo}. Here we provide a somehow alternative argument which gives that case as well as the desired end-point estimates. Our first observation is that if $\gamma_{p_0,r_0}=0$ there is nothing to prove. Indeed this only happens when $p_0=1$, $r_0=\infty$ and hence \eqref{off-extrapol:compat} gives that $\frac1{p}=1+\frac1r$ which contradicts the facts that $p>1$ and $r<\infty$. Thus, from now on we assume that $\gamma_{p_0,r_0}>0$ and observe that  \eqref{off-extrapol:compat} allows us to easily write $\gamma:=\gamma_{p,r}=\gamma_{p_0,r_0}$.

\medskip

\textbf{Case 1:} $\frac1s:=\frac1p-\frac1{p_0}=\frac1q-\frac1{q_0}=\frac1r-\frac1{r_0}>0$, hence
$1<p<p_0\le\infty$, $0<q<q_0\le\infty$ and $0<r<r_0\le\infty$ (note that $p_0$ and/or $r_0$ and/or $q_0$ could be infinity).

Fix $(f,g)\in\mathscr{F}$ and $w\in A_{p,r}$. As observed above, one has $w^{-p'}\in A_{p'\gamma}$. Also, $p'\gamma=1+\frac{p'}{r}>1$ since $r<r_0$, hence if we set $M'h:=M(hw^{-p'})\,w^{p'}$ where $M$ is the Hardy-Littlewood maximal function one has that $M'$ is bounded on $L^{(p'\gamma)'}(w^{-p'})$, write $\|M'\|$ to denote its operator norm and introduce the Rubio de Francia algorithm
\[
\mathcal{R}'h=\sum_{k=0}^\infty \frac{(M')^{(k)} h}{2^k\,\|M'\|^k},
\]
where for $k\ge 1$, we write $(M')^{(k)}= M'\circ\dots\circ M'$ to denote $k$ iterations of $M'$ and for $k=0$ is the identity operator .
Based on this definition (see \cite{CMP} for more details) one can readily see that if $0\le h\in L^{(p'\gamma)'}(w^{-p'})$ then
\begin{equation}\label{R':1}
h\le \mathcal{R}'h,
\quad
\|\mathcal{R}'h\|_{L^{(p'\gamma)'}(w^{-p'})}\le 2\|h\|_{L^{(p'\gamma)'}(w^{-p'})},
\quad
[\mathcal{R}'h\,w^{-p'}]_{A_1}\le 2\|M'\|.
\end{equation}
Notice that in particular $0<\mathcal{R}'h<\infty$ a.e. if $h$ is non-trivial.

Next let us  observe that without loss of generality we may assume that
$0<\|gw\|_{L^{p}}<\infty$: if $\|gw\|_{L^{p}}=\infty$ there would be nothing to prove and if $\|gw\|_{L^{p}}=0$ then $g=0$ a.e. and by \eqref{off-extrapol:hyp} we would get that $f=0$ a.e. which would trivially yield the desired estimate. This implies that the auxiliary function $0\le h=gw^{p'}/\|gw\|_{L^p}$ clearly satisfies $\|h\|_{L^p(w^{-p'})}=1$. Set $H=\mathcal{R}'(h^{\frac{p}{(p'\gamma)'}})^{\frac{(p'\gamma)'}{p}}= \mathcal{R}'(h^{\frac{p}{r\gamma}})^{\frac{r\gamma}{p}} $ which satisfies that $0<H<\infty$ a.e. and by \eqref{R':1}
\begin{equation}\label{R':2}
h\le H,
\qquad
\|H\|_{L^{p}(w^{-p'})}\lesssim \|h\|_{L^{p}(w^{-p'})}=1,
\qquad
[H^{\frac{p}{r\gamma}}w^{-p'}]_{A_1}\le 2\|M'\|.
\end{equation}
Set $W=H^{-\frac{p}{s}}w^{1+\frac{p'}{s}}$ and we claim that $W\in A_{p_0,r_0}$. Assuming this momentarily we note that \eqref{off-extrapol:compat} and the definition of $s$ gives $\frac1q=\frac1{q_0}+\frac1s$. Hence we can use H\"older's inequality, \eqref{off-extrapol:hyp}, and \eqref{R':2} to obtain
\begin{multline*}
\|fw\|_{L^q}
=
\|(fW)\,(H^{\frac{p}{s}}w^{-\frac{p'}{s}})\|_{L^q}
\le
\|fW\|_{L^{q_0}}\,\|H^{\frac{p}{s}}w^{-\frac{p'}{s}}\|_{L^s}
\lesssim
\|g W\|_{L^{p_0}}
\\
\lesssim
\|gw\|_{L^p} \|Hw^{-p'}W\|_{L^{p_0}}
=
\|gw\|_{L^p} \|H^{1-\frac{p}{s}}w^{-p'(\frac1p-\frac1s)}\|_{L^{p_0}}
\lesssim
\|gw\|_{L^p},
\end{multline*}
where the last estimate is trivial when $p_0=\infty$ since $s=p$, and otherwise follows from \eqref{R':2} and the fact that $1-\frac{p}{s}=\frac{p}{p_0}$.

This gives the desired estimate and therefore it remains to show that $W=H^{-\frac{p}{s}}w^{1+\frac{p'}{s}}\in A_{p_0,r_0}$. If $r_0=\infty$, one can easily see that \eqref{off-extrapol:conc} and the definition of $s$ give that $s=r$ and
\[
\gamma
=
1+\frac1r-\frac1p
=
1-\frac1{p_0}=
\frac1{p_0'},
\qquad
1+\frac{p'}{s}
=
p'\left(\frac1{p'}+\frac1s\right)=
p'\left(1-\frac1{p_0}\right)=
\frac{p'}{p_0'}.
\]
Hence $W^{-p_0'}=H^{\frac{p}{r\gamma}}w^{-p'}$ and the last condition  in \eqref{R':2} readily yields $W^{-p_0'}\in A_1$. That is, $W\in A_{p_0,\infty}$ as desired.

Consider next the case $0<r_0<\infty$ and note that definition of $s$ implies that
\[
1+\frac{p'}s-\frac{p'r\gamma}{s}
=
1+\frac{p'}s-\frac{p'+r}{s}
=
\frac1{r}\left(\frac1r-\frac1s\right)
=\frac{r}{r_0}.
\]
On the other hand,  \eqref{off-extrapol:conc}  and the definition of $s$ also yield $\gamma=\frac1{r_0}+\frac1{p_0'}$ and
\begin{align*}
1-\frac{p_0'r\gamma}{s}
=
r\left(\frac1r-\frac{p_0'\gamma}{s}\right)
=
r\left(\frac1{r_0}-\frac{p_0'\gamma-1}{s}\right)
=
\frac{rp_0'}{r_0}\left(\frac1{p_0'}-\frac{1}{s}\right)
=
\frac{rp_0'}{r_0p'}.
\end{align*}
Note that the right-hand side  is positive and strictly smaller than $1$ since in this case $0<r<r_0<\infty$ and $1<p<p_0$. As a consequence, $0<\frac{p_0'r\gamma}{s}<1$ and $(\frac{s}{p_0'r\gamma})'= \frac{r_0p'}{rp_0'}$. We can then use H\"older's inequality and the previous calculations to conclude that
\begin{multline*}
\left(\dashint_{Q} W^{-p_0'}\,dx\right)^{\frac1{p_0'}}
=
\left(\dashint_{Q} \big(H^{\frac{p}{r\gamma}}w^{-p'}\big)^{\frac{p_0'r\gamma}{s}}\,w^{-\frac{p_0'r}{r_0}}dx\right)^{\frac1{p_0'}}
\\
\le
\left(\dashint_{Q} H^{\frac{p}{r\gamma}} w^{-p'}\,dx\right)^{\frac{r\gamma}{s}}
\left(\dashint_{Q} w^{-p'}\,dx\right)^{\frac{r}{r_0 p'}}
\\
\le
[H^{\frac{p}{r\gamma}} w^{-p'}]_{A_1}^{\frac{r\gamma}{s}}\essinf_{Q}(H^{\frac{p}{r\gamma}} w^{-p'})^{\frac{r\gamma}{s}}
\left(\dashint_{Q} w^{-p'}\,dx\right)^{\frac{r}{r_0 p'}}
.
\end{multline*}
This yields
\begin{align*}
&\left(\dashint_{Q} W^{r_0}\,dx\right)^{\frac1{r_0}}\left(\dashint_{Q} W^{-p_0'}\,dx\right)^{\frac1{p_0'}}
\\
&\qquad\le
[H^{\frac{p}{r\gamma}} w^{-p'}]_{A_1}^{\frac{r\gamma}{s}}
\left(\dashint_{Q} H^{\frac{r_0p}{s}} w^{-\frac{p'r_0 r\gamma}{s} }W^{r_0}\,dx\right)^{\frac1{r_0}}
\left(\dashint_{Q} w^{-p'}\,dx\right)^{\frac{r}{r_0 p'}}
\\
&\qquad=
[H^{\frac{p}{r\gamma}} w^{-p'}]_{A_1}^{\frac{r\gamma}{s}}
\left(\dashint_{Q} w^{r}\,dx\right)^{\frac1{r_0}}
\left(\dashint_{Q} w^{-p'}\,dx\right)^{\frac{r}{r_0p'}}
\\
&\qquad\le
[H^{\frac{p}{r\gamma}} w^{-p'}]_{A_1}^{\frac{r\gamma}{s}}\,[w]_{A_{p,r}}^{\frac{r}{r_0}}.
\end{align*}
Using \eqref{R':2} we conclude as desired that $W\in A_{p_0,r_0}$ and the proof of the present case is complete.

\medskip

\textbf{Case 2:} $0<q<\infty$  and $\frac1s:=\frac1{p_0}-\frac1p=\frac1{q_0}-\frac1q=\frac1{r_0}-\frac1r>0$, hence
$1\le p_0<p\le\infty$, $0<q_0<q<\infty$ and $0<r_0<r<\infty$ (note that $p$ could be infinity).

Fix $(f,g)\in\mathscr{F}$ and $w\in A_{p,r}$.  As observed above, one has $w^{-p'}\in A_{p'\gamma}$. Additionally, $p'\gamma=1+\frac{p'}r>1$ since $r<\infty$, hence $M$ is bounded on $L^{p'\gamma}(w^{-p'})$ and writing $\|M\|$ for its operator norm one can then introduce the Rubio de Francia algorithm
\[
\mathcal{R}h=\sum_{k=0}^\infty \frac{M^{(k)} h}{2^k\,\|M\|^k}.
\]
This (see \cite{CMP} for more details) readily gives that if $0\le h\in L^{p'\gamma}(w^{-p'})$ then
\begin{equation}\label{R:1}
h\le \mathcal{R}h,
\qquad
\|\mathcal{R}h\|_{L^{p'\gamma}(w^{-p'})}\le 2\|h\|_{L^{p'\gamma}(w^{-p'})},
\qquad
[\mathcal{R}h]_{A_1}\le 2\|M\|.
\end{equation}

On the other hand, notice that the definition of $s$ and \eqref{off-extrapol:compat} give $\frac1{q_0}=\frac1q+\frac1s$, hence $\frac{q}{q_0}=(\frac{s}{q_0})'$. Then  by duality there exists $0\le h\in L^{\frac{s}{q_0}}(w^q)$ with $\|h\|_{L^{\frac{s}{q_0}}(w^q)}=1$ such that
\[
\|fw\|_{L^q}=\|f^{q_0}\|_{L^{\frac{q}{q_0}}(w^q)}^{\frac1{q_0}}
=
\left(\int f^{q_0} h w^q\,dx\right)^{\frac1{q_0}}.
\]
Set $H=\mathcal{R}\Big(h^{\frac{s}{p'q_0\gamma}} w^{\frac{p'+q}{p'\gamma}}\Big)^{\frac{p'q_0\gamma}{s}}\,w^{-\frac{(p'+q)q_0}{s}}$ which satisfies by \eqref{R:1}
\begin{equation}\label{R:2}
h\le H,
\qquad
\|H\|_{L^{\frac{s}{q_0}}(w^{q})}\lesssim \|h\|_{L^{\frac{s}{q_0}}(w^{q})}=1,
\qquad
[H^{\frac{s}{p'q_0\gamma}} w^{\frac{p'+q}{p'\gamma}}]_{A_1}\le 2\|M\|.
\end{equation}
Let $W=H^{\frac1{q_0}}w^{\frac{q}{q_0}}$ and claim that $W\in A_{p_0,r_0}$. Assuming this momentarily we conclude that
\begin{multline*}
\|fw\|_{L^q}
=
\Big(\int f^{q_0} h w^q\,dx\Big)^{\frac1{q_0}}
=
\|fW\|_{L^{q_0}}
\lesssim
\|gW\|_{L^{p_0}}
=
\|(gw) (H^{\frac1{q_0}}w^{\frac{q}{q_0}-1})\|_{L^{p_0}}
\\
\le
\|gw\|_{L^{p}} \|H^{\frac1{q_0}}w^{\frac{q}{q_0}-1}\|_{L^{s}}
=
\|gw\|_{L^p}\|H\|_{L^{\frac{s}{q_0}}(w^q)}^{\frac1{q_0}}
\lesssim
\|gw\|_{L^p},
\end{multline*}
where we have used \eqref{R:2}, that $\frac1{p_0}=\frac1p+\frac1s$, H\"older's inequality and the fact that $s(\frac{q}{q_0}-1)=q$ by \eqref{off-extrapol:compat} and our choice of $s$.

To complete this case we just need to see that $W=H^{\frac1{q_0}}w^{\frac{q}{q_0}} \in A_{p_0,r_0}$. When $p_0=1$ one has $\gamma=r_0^{-1}$,  $s=p'$. Hence $\frac{s}{p' q_0 \gamma}=\frac{r_0}{q_0}$, and \eqref{off-extrapol:compat} yield
\[
\frac{p'+q}{p'\gamma}=r_0q\left(\frac1q+\frac1{p'}\right)
=
\frac{p'+q}{p'\gamma}=\frac{r_0q}{q_0},
\]
As a result,
$W^{r_0}=H^{\frac{s}{p'q_0\gamma}}w^{\frac{p'+q}{p'\gamma}}$ and \eqref{R:2} readily gives that $W^{r_0}\in A_1$, that is, $W\in A_{1,r_0}$ as desired.

On the other hand, if $p_0>1$ we observe that by \eqref{off-extrapol:compat} and the definition of $s$
\[
\frac{p'}{p_0'}=p'\left(1-\frac1{p_0}\right)
=p'\left(\frac1{p'}-\frac1{s}\right)
=
\frac{q}{q_0}-\frac{q}{s}-\frac{p'}{s}
=
\frac{q}{q_0}-\frac{p'+q}{s}
\]
and also that
\begin{multline*}
1-\frac{r_0p' \gamma}{s}
=
r_0\left(\frac1{r_0}-\frac{p'\gamma}{s}\right)
=
r_0\left(\frac1{r}+\frac1{s}-\frac{p'\gamma}{s}\right)
=
r_0\left(\frac1{r}-\frac{p'\gamma-1}{s}\right)
\\
=
r_0\left(\frac1{r}-\frac{p'}{rs}\right)
=
\frac{r_0p'}{r}\left(\frac1{p'}-\frac1s\right)
=
\frac{r_0p'}{rp_0'}.
\end{multline*}
Note that the right-hand side  of the last displayed equation is positive and  strictly smaller than $1$ since  $r_0<r<\infty$ and $1<p_0<p$. As a consequence, $0<\frac{r_0p' \gamma}{s}<1$ and $(\frac{s}{r_0 p'\gamma})'= \frac{rp_0'}{r_0p'}$. We can then use H\"older's inequality and the previous calculations to conclude that
\begin{multline*}
\left(\dashint_{Q} W^{r_0}\,dx\right)^{\frac1{r_0}}
=
\left(\dashint_{Q} H^{\frac{r_0}{q_0}}\,w^{\frac{r_0 q}{q_0}}\,dx\right)^{\frac1{r_0}}
=
\left(\dashint_{Q} \big(H^{\frac{s}{p'q_0\gamma}}w^{\frac{p'+q}{p'\gamma}}\big)^{\frac{r_0p'\gamma}{s}} \,w^{\frac{r_0p'}{p_0'}}\,dx\right)^{\frac1{r_0}}
\\
\le
\left(\dashint_{Q} H^{\frac{s}{p'q_0\gamma}}w^{\frac{p'+q}{p'\gamma}}\,dx\right)^{\frac{p'\gamma}{s}}
\left(\dashint_{Q} w^{r}\,dx\right)^{\frac{p'}{rp_0'}}
\\
\le
[H^{\frac{s}{p'q_0\gamma}}w^{\frac{p'+q}{p'\gamma}}]_{A_1}^{\frac{p'\gamma}{s}} \essinf_{Q} (H^{\frac{s}{p'q_0\gamma}}w^{\frac{p'+q}{p'\gamma}})^{\frac{p'\gamma}{s}} \left(\dashint_{Q} w^{r}\,dx\right)^{\frac{p'}{rp_0'}}.
\end{multline*}
This and \eqref{off-extrapol:compat} imply
\begin{align*}
&\left(\dashint_{Q} W^{r_0}\,dx\right)^{\frac1{r_0}} \left(\dashint_{Q} W^{-p_0'}\,dx\right)^{\frac1{p_0'}}
\\
&\qquad\le
[H^{\frac{s}{p'q_0\gamma}}w^{\frac{p'+q}{p'\gamma}}]_{A_1}^{\frac{p'\gamma}{s}}
\left(\dashint_{Q} w^{r}\,dx\right)^{\frac{p'}{rp_0'}}
\left(\dashint_{Q} H^{\frac{p_0'}{q_0}} w^{\frac{(p'+q)p_0'}{s} }W^{-p_0'}\,dx\right)^{\frac1{p_0'}}
\\
&\qquad=
[H^{\frac{s}{p'q_0\gamma}}w^{\frac{p'+q}{p'\gamma}}]_{A_1}^{\frac{p'\gamma}{s}}
\left(\dashint_{Q} w^{r}\,dx\right)^{\frac{p'}{rp_0'}}
\left(\dashint_{Q} w^{-p'}\,dx\right)^{\frac{1}{p_0'}}
\\
&\qquad\le
[H^{\frac{s}{p'q_0\gamma}}w^{\frac{p'+q}{p'\gamma}}]_{A_1}^{\frac{p'\gamma}{s}}\,[w]_{A_{\infty,r}}^{\frac{p'}{p_0'}}.
\end{align*}
Taking the supremum over all cubes and using \eqref{R:2} we obtain that $W\in A_{p_0,r_0}$.

\medskip

\textbf{Case 3:} $q=\infty$ and $\frac1s:=\frac1{p_0}-\frac1p=\frac1{q_0}=\frac1{r_0}-\frac1r>0$, hence $1\le p_0<p\le\infty$,
$0<q_0<\infty$, and $0<r_0<r<\infty$  (note that $p$ could also be infinity).

Again, fix $(f,g)\in\mathscr{F}$ and $w\in A_{p,r}$. We may assume that $\|gw\|_{L^p}<\infty$, otherwise there is nothing to prove.  Note that as above one has $w^{-p'}\in A_{p'\gamma}$. Additionally, $p'\gamma=1+\frac{p'}r>1$ since $r<\infty$, hence $M$ is bounded on $L^{p'\gamma}(w^{-p'})$ and writing $\|M\|$ for its operator norm one can then introduce the Rubio de Francia algorithm
\[
\mathcal{R}h=\sum_{k=0}^\infty \frac{M^{(k)} h}{2^k\,\|M\|^k}.
\]
This (see \cite{CMP} for more details) readily gives that if $0\le h\in L^{p'\gamma}(w^{-p'})$ then
\begin{equation}\label{R:3}
h\le \mathcal{R}h,
\qquad
\|\mathcal{R}h\|_{L^{p'\gamma}(w^{-p'})}\le 2\|h\|_{L^{p'\gamma}(w^{-p'})},
\qquad
[\mathcal{R}h]_{A_1}\le 2\|M\|.
\end{equation}

We fix $x_0\in \re^n$ and $0<\tau_0<\infty$ and let $h:=h_{x_0,\tau_0}:=|B(x_0,\tau_0)|^{-1}\mathbf{1}_{B(x_0, r_0)}$ which is a non-negative function satisfying $\|h\|_{L^1}=1$. Set $H:=\mathcal{R}\Big(h^{\frac{1}{p'\gamma}} w^{\frac{1}{\gamma}}\Big)^{p'\gamma}\,w^{-p'}$ which satisfies by \eqref{R:3}
\begin{equation}\label{R:4}
h\le H,
\qquad
\|H\|_{L^{1}}\lesssim \|h\|_{L^1}=1,
\qquad
[H^{\frac{1}{p'\gamma}} w^{\frac{1}{\gamma}}]_{A_1}\le 2\|M\|.
\end{equation}
Set $W:=H^{\frac1{q_0}}w$ and claim that $W\in A_{p_0,r_0}$. Assuming this momentarily we conclude that
\begin{multline*}
\left(\dashint_{B(x_0,\tau_0)} (f(x)w(x))^{q_0}dx\right)^{\frac1{q_0}}
=\Big(\int f^{q_0} h w^{q_0}\,dx\Big)^{\frac1{q_0}}
\le
\|fW\|_{L^{q_0}}
\lesssim
\|gW\|_{L^{p_0}}
\\
=
\|(gw) H^{\frac1{q_0}}\|_{L^{p_0}}
\le
\|gw\|_{L^{p}} \|H^{\frac1{q_0}}\|_{L^{q_0}}
\lesssim
\|gw\|_{L^p},
\end{multline*}
where we have used \eqref{R:4}, that $\frac1{p_0}=\frac1p+\frac1{q_0}$ and H\"older's inequality. Note that the implicit constants are all independent of $x_0$ and $\tau$. Hence, since we have that $\|gw\|_{L^p}<\infty$, we conclude that $fw\in L^{q_0}_{\rm loc}(\re^n)$. Thus, if we write $E_{fw}$ for the Lebesgue points of $(fw)^{q_0}$ we have that $|\re^n\setminus E_{fw}|=0$. Moreover if we let $x_0\in E_{fw}$ we conclude that
\[
f(x_0)w(x_0)=\lim_{\tau\to 0^+}\left(\dashint_{B(x_0,\tau_0)} (f(x)w(x))^{q_0}dx\right)^{\frac1{q_0}}
\lesssim
\|gw\|_{L^p},
\]
This eventually shows that $fw\in L^\infty(\re^n)$ with $\|fw\|_{L^\infty}\lesssim\|gw\|_{L^p}$ which is our desired estimate.

We are left with showing that $W:=H^{\frac1{q_0}}w\in A_{p_0,r_0}$. Consider first the case $p_0=1$. In this case we have that \eqref{off-extrapol:compat} implies that $p'=q_0$, $\gamma=r_0^{-1}$ and $\frac{r_0}{q_0}=\frac1{p' \gamma}$.
Thus, $ W^{r_0}=H^{\frac{1}{p'\gamma}}w^{\frac{1}{\gamma}}$ and \eqref{R:4} readily implies that $W^{r_0}\in A_1$, that is, $W\in A_{1,r_0}$ as desired.

Consider next the case $p_0>1$ and observe that \eqref{off-extrapol:compat} implies  $\frac1{p'}-\frac1{q_0}=\frac1{p_0'}$ and
\begin{multline*}
1-\frac{r_0p' \gamma}{q_0}
=
r_0\left(\frac1{r_0}-\frac{p'\gamma}{q_0}\right)
=
r_0\left(\frac1{r}+\frac1{q_0}-\frac{p'\gamma}{q_0}\right)
\\
=
r_0\left(\frac1{r}-\frac{p'\gamma-1}{q_0}\right)
=
\frac{r_0p'}{r}\left(\frac1{p'}-\frac1{q_0}\right)
=
\frac{r_0p'}{rp_0'}.
\end{multline*}
Note that the right-hand side  is positive and strictly smaller than $1$ since in this case $r_0<r$ and $1< p_0<p$. As a consequence, $0<\frac{r_0p' \gamma}{q_0}\le 1$ and $(\frac{q_0}{r_0 p'\gamma})'= \frac{rp_0'}{r_0p'}$. Using this and H\"older's inequality we obtain
\begin{multline*}
\left(\dashint_{Q} W^{r_0}\,dx\right)^{\frac1{r_0}}
=
\left(\dashint_{Q} H^{\frac{r_0}{q_0}}\,w^{r_0}\,dx\right)^{\frac1{r_0}}
=
\left(\dashint_{Q} \big(H^{\frac{1}{p'\gamma}}w^{\frac{1}{\gamma}}\big)^{\frac{r_0p'\gamma}{q_0}} \,w^{\frac{r_0p'}{p_0'}}\,dx\right)^{\frac1{r_0}}
\\
\le
\left(\dashint_{Q} H^{\frac{1}{p'\gamma}}w^{\frac{1}{\gamma}}\,dx\right)^{\frac{p'\gamma}{q_0}}
\left(\dashint_{Q} w^{r}\,dx\right)^{\frac{p'}{rp_0'}}
\\
\le
[H^{\frac{1}{p'\gamma}}w^{\frac{1}{\gamma}}]_{A_1}^{\frac{p'\gamma}{q_0}} \essinf_{Q} (H^{\frac{1}{p'\gamma}}w^{\frac{1}{\gamma}})^{\frac{p'\gamma}{q_0}} \left(\dashint_{Q} w^{r}\,dx\right)^{\frac{p'}{rp_0'}}.
\end{multline*}
This and \eqref{off-extrapol:compat} imply
\begin{align*}
&\left(\dashint_{Q} W^{r_0}\,dx\right)^{\frac1{r_0}} \left(\dashint_{Q} W^{-p_0'}\,dx\right)^{\frac1{p_0'}}
\\
&\qquad\le
[H^{\frac{1}{p'\gamma}}w^{\frac{1}{\gamma}}]_{A_1}^{\frac{p'\gamma}{q_0}}
\left(\dashint_{Q} w^{r}\,dx\right)^{\frac{p'}{rp_0'}}
\left(\dashint_{Q} H^{\frac{p_0'}{q_0}} w^{\frac{p'p_0'}{q_0} }W^{-p_0'}\,dx\right)^{\frac1{p_0'}}
\\
&\qquad=
[H^{\frac{1}{p'\gamma}}w^{\frac{1}{\gamma}}]_{A_1}^{\frac{p'\gamma}{q_0}}
\left(\dashint_{Q} w^{r}\,dx\right)^{\frac{p'}{rp_0'}}
\left(\dashint_{Q} w^{-p'}\,dx\right)^{\frac{1}{p_0'}}
\\
&\qquad\le
[H^{\frac{1}{p'\gamma}}w^{\frac{1}{\gamma}}]_{A_1}^{\frac{p'\gamma}{q_0}}\,[w]_{A_{p,r}}^{\frac{p'}{p_0'}}.
\end{align*}
Taking the supremum over all cubes and using \eqref{R:4} we obtain that $W\in A_{p_0,r_0}$. This completes the proof.
\end{proof}

\begin{Remark}
A careful read of the above argument reveals it works in the setting of spaces of homogeneous type, as there we have a theory of Muckenhoupt weights much as in the case of the Euclidean setting. Also, the argument can be easily extended to Muckenhoupt bases (see for instance \cite[Chapter 3]{CMP}). The proof can be easily adapted to that setting. Further details are left to the interested reader.
\end{Remark}

\subsection{Proof of Theorem \ref{theor:extrapol-general}}\label{section:proof-main}

To prove our main result we need some characterization of the $A_{\vec p, \vec r}$ classes. The following result extends \cite[Lemma 3.2]{LMO} to the end-point cases.
The proof can be carried out  \textit{mutatis mutandis} (with the appropriate changes of notation) and we leave the details to the interested reader.

\begin{Lemma}\label{lemma:main}
Consider $\vec p=(p_1,\dots, p_m)$ with $1\le p_1,\dots,p_m\le \infty$ and $\vec{r}=(r_1,\dots,r_{m+1})$
	with $1\le r_1,\dots,r_{m+1}< \infty$ such that $\vec{r}\preceq_\star\vec{p}$. Using the previous notation we assume that
	\begin{equation}\label{def:varrho}
	\frac1{\varrho}:=\frac 1{r_m}-\frac 1{r_{m+1}'}+\sum_{i=1}^{m-1}\frac1{p_i}
	=
	\frac1{\delta_m}+\frac1{\delta_{m+1}}
	>0,
	\end{equation}
	and for every $1\le i\le m-1$
	\begin{equation}\label{def:theta}
	\frac1{\theta_i}
	:=
	\frac{1-r}r-\frac1{\delta_i}
	=
	\left(\sum_{j=1}^{m+1} \frac1{\delta_j}\right)-\frac1{\delta_i}
	>0.
	\end{equation}
	Then the following hold:
	\begin{list}{$(\theenumi)$}{\usecounter{enumi}\leftmargin=.8cm
			\labelwidth=.8cm\itemsep=0.2cm\topsep=.1cm
			\renewcommand{\theenumi}{\roman{enumi}}}
		
		\item Given  $\vec{w}=(w_1,\dots,w_m)\in A_{\vec p, \vec r}$, write $w:=\prod_{i=1}^mw_i $ and
		set
		\begin{equation}\label{formula-1}
		\widehat{w}:=\Big(\prod_{i=1}^{m-1}w_i \Big)^{\varrho}
		\qquad
		\mbox{and}
		\qquad
		W
		:=
		w^{ r_m } \widehat{w}^{- \frac{r_m}{\delta_{m+1}}}
		=
		w_m^{ r_m }\widehat{w}^{\frac{r_m}{\delta_m}}
		\end{equation}
		Then,
		\begin{list}{$(\theenumi.\theenumii)$}{\usecounter{enumii}\leftmargin=.4cm
				\labelwidth=.8cm\itemsep=0.2cm\topsep=.1cm
				\renewcommand{\theenumii}{\arabic{enumii}}}

			\item $w_i^{ \theta_i  }\in A_{\frac{1-r}{r}\theta_i }$ with $\Big[w_i^{ \theta_i  }\Big]_{ A_{\frac{1-r}{r}\theta_i }} \le [\vec w]_{A_{\vec p, \vec r}}^{\theta_i }$, for every $1\le i\le m-1$.
			
			\item $\widehat{w}\in A_{\frac{1-r}{r}\varrho}$ with
			$[\widehat{w}]_{A_{\frac{1-r}{r}\varrho}}\le [\vec w]_{A_{\vec p, \vec r}}^{\varrho}$.
			
			\item $W\in A_{\frac{p_m}{r_m},\frac{\delta_{m+1}}{r_m}}(\widehat{w})$
			with $[W]_{A_{\frac{p_m}{r_m},\frac{\delta_{m+1}}{r_m}} (\widehat{w})}
			\le
			[\vec w]_{A_{\vec p, \vec r}}^{r_m}$.
		\end{list}

		\item Given $w_i^{ \theta_i }\in A_{\frac{1-r}{r}\theta_i}$,  $1\le i\le m-1$,
		such  that
		\begin{equation}\label{formula2-1}
		\widehat{w}=
		\Big(\prod_{i=1}^{m-1}w_i \Big)^{\varrho}
		\in A_{\frac{1-r}{r}\varrho}
		\end{equation}
		and
		$W\in A_{\frac{p_m}{r_m},\frac{\delta_{m+1}}{r_m}} (\widehat{w})$, let us set
		\begin{equation}\label{formula-2}
		w_m
		:=
		W^{\frac{1}{r_m}} \widehat{w}^{-\frac{1}{\delta_m}}.
		\end{equation}
		Then $\vec{w}=(w_1,\dots, w_m)\in A_{\vec p, \vec r}$ and, moreover,
		\[
		[\vec w]_{A_{\vec p, \vec r}}
		\le
		[W]_{A_{\frac{p_m}{r_m},\frac{\delta_{m+1}}{r_m}} (\widehat{w})}^{\frac1{r_m}}
		[\widehat{w}]_{A_{\frac{1-r}{r}\varrho}}^{\frac1 {\varrho}}
		\prod_{i=1}^{m-1}\Big[w_i^{ \theta_i }\Big]_{A_{\frac{1-r}{r}\theta_i}}^{\frac1{\theta_i}}.
		\]

		\item For any measurable function $f\ge0 $ and in the context of $(i)$ or $(ii)$ there hold
		\begin{equation}\label{LHS-rew}
		\|fw\|_{L^p }
		=
		\Big\|\Big(f\widehat{w}^{-\frac1{r_{m+1}'}}\Big)^{r_m} W\Big\|_{L^\frac{p}{r_m}(d\widehat{w})}^{\frac1{r_m}}
		\end{equation}
		and
		\begin{equation}\label{RHS-rew}
		\|fw_m\|_{L^{p_m} }
		=
		\Big\|\Big(f\widehat{w}^{-\frac1{r_m}}\Big)^{r_m}W\Big\|_{L^\frac{p_m}{r_m}(d\widehat{w})}^{\frac1{r_m}}.
		\end{equation}
	\end{list}
\end{Lemma}

\medskip

In the previous result we are assuming that $\delta_{m+1}^{-1}$ and $\delta_{m}^{-1}$ cannot be simultaneously zero, but they could be zero individually, in which case some of the statements require appropriate interpretations which are left to the interest reader.  We also note that the assumption $\varrho^{-1}>0$, leads to the fact that the class $A_{\frac{p_m}{r_m},\frac{\delta_{m+1}}{r_m}} (\widehat{w})$ is non-trivial ---recall that, as observed above, $W\in A_{1,\infty}(\widehat{w})$ means that $W\approx 1$ almost everywhere.

\medskip

\begin{proof}[Proof of Theorem \ref{theor:extrapol-general}]
	The argument is almost identical to that in \cite[Proof of Theorem 1.1]{LMO} with the difference that we allow the $p_i$'s or $q_i$'s to be infinity,
	and this requires us to use Theorem \ref{thm:offdiagonal} in place of \cite[Theorem 3.1]{LMO} (which is really \cite[Theorem 5.1]{Duo}) and Lemma \ref{lemma:main} in place of \cite[Lemma 3.2]{LMO}. There is however a subtle point that we would like to address here. In the present scenario, to use Lemma \ref{lemma:main} we need to assume that $\varrho^{-1}>0$, while in \cite[Lemma 3.2]{LMO} this condition follows automatically from the assumption that $\delta_{m+1}^{-1}>0$ since there it was required that $p>r_{m+1}'$ (using the notation there $\vec{r}\preceq\vec{p}$ ). Thus we consider two possible scenarios $r_{m+1}'>p$ or $r_{m+1}'=p$.
	
	Assume first that  $r_{m+1}'>p$ (that is, $\vec{r}\preceq\vec{p}$), hence $\delta_{m+1}^{-1}>0$.  We just follow \cite[Proof of Theorem 1.1]{LMO}, an argument which changes one exponent at a time, with the alterations pointed out above after noting that implicit in that scheme it is shown that the intermediate exponents $\vec{t}=(t_1,\dots, t_m)$ satisfy $\vec{r}\preceq\vec{t}$, that is,  $r_{m+1}'>t=(\sum_{i=1}^m\frac1{t_i})^{-1}$ since in the present case
	$\vec{r}\preceq\vec{p}$ and $\vec{r}\prec\vec{q}$. This means that in the successive uses of Lemma \ref{lemma:main} we always have that the corresponding $\varrho^{-1}>0$ and in the iteration argument we never reach the end-point $r_{m+1}'$ in the target space.
	
	The second scenario is that on which $p=r_{m+1}'$, that is,  $\delta_{m+1}^{-1}=0$. Let $\vec q=(q_1,\dots,q_m)$ satisfy $\vec{r}\prec\vec{q}$ so that $p=r_{m+1}'>q$. Let $\mathcal{I}=\{i: 1\le i\le m, p_i>q_i\}$ and observe that $\mathcal{I}\neq\emptyset$ since otherwise we get a contradiction:
	\[
	\frac1{r_{m+1}'}=\frac1{p}=\sum_{i=1}^m \frac1{p_i}\ge \sum_{i=1}^m \frac1{q_i}=\frac1{q}>\frac1{r_{m+1}'}.
	\]
	Thus rearranging the $f_i$'s if needed we may assume that $p_m>q_m>r_m$ ---since $\vec{r}\prec\vec{q}$. In particular $\delta_m^{-1}>0$ and Lemma \ref{lemma:main} applies. This allows us to follow \textbf{Step 1} in  \cite[Proof of Theorem 1.1]{LMO} \textit{mutatis mutandis} with the modifications pointed out above to obtain the desired estimate for the exponent $\vec{t}=(t_1,\dots, t_m)$ with $t_i=p_i$ for $1\le i\le m$ and $t_m=q_m$ and the associated class of weights $A_{\vec{t},\vec{r}}$. Moreover,
	\[
	\frac1t
	=
	\sum_{i=1}^{m} \frac1{t_i}
	=
	\sum_{i=1}^{m-1} \frac1{p_i} +\frac1{q_m}
	>
	\sum_{i=1}^{m} \frac1{p_i}
	=
	\frac1{p}
	=
	\frac1{r_{m+1}'}.
	\]
	Hence $r_{m+1}'>t$, that is, $\vec{r}\preceq\vec{t}$  and we can extrapolate from this exponent to change the other entries in $\vec{t}$ much as in the first scenario above. Let us observe that, again, at any step in the iteration we will never reach the end-point $r_{m+1}'$ in the target space. Further details are left to the interested reader.

	To complete the proof we sketch how to establish \eqref{extrapol:vv}. The proof is almost identical to \cite[Proof of Theorem 1.1]{LMO}, the only difference is that here some of the $s_i$'s could be infinity. If that is the case we just need to observe that by assumption not all the $s_i$ can be infinity (otherwise $s=0$), hence
\begin{multline*}
\Big(\sum_j \prod_{i=1}^{m} \|f_i^j v_i\|_{L^{s_i}}^s\Big)^{\frac1s}
\le
\prod_{i:s_i=\infty}  \Big\|\{f_i^j v_i\}_j\Big\|_{L^{\infty}_{\ell^\infty}}
\Big(\sum_j \prod_{i: s_i\neq \infty} \|f_i^j v_i\|_{L^{s_i}}^s\Big)^{\frac1s}
\\
\le
\prod_{i:s_i=\infty}  \Big\|\{f_i^j v_i\}_j\Big\|_{L^{\infty}_{\ell^\infty}}
\prod_{i: s_i\neq \infty} \Big(\sum_j \|f_i^j v_i\|_{L^{s_i}}^{s_i}\Big)^{\frac1{s_i}}
=
\prod_{i=1}^m  \Big\|\{f_i^j v_i\}_j\Big\|_{L^{s_i}_{\ell^{s_i}}}.
\end{multline*}
With this estimate in hand the proof can be completed in exactly the same manner and we omit the details.

\end{proof}

\section{Calder\'on-Zygmund operators, the bilinear Hilbert transform, and sparse forms}\label{sec:apps}

In \cite{LMO} several applications were given to show that extrapolation can be used to provide almost trivial proofs of known results and also to obtain new estimates. With our main result in this paper we can easily complete the picture and obtain the end-point cases. Here we will just indicate the resulting estimates, see \cite{LMO} and the references therein for more details and the precise definitions.

We start with $T$ being a multilinear Calder\'on-Zygmund operator. Applying Theorem \ref{theor:extrapol-general:CZO} together with the weighted estimates from \cite{LOPTT} one immediately obtains \eqref{extrapol:vv*}. This should be compared with \cite[Section 2.1]{LMO} where all the $q_i$'s and $s_i$'s are assumed to be finite.

Our second example is as follows. Fix $\vec{r}=(r_1,\dots,r_{m+1})$, with  $r_i\ge 1$ for $1\le i\le m+1$, and $\sum_{i=1}^{m+1}\frac 1{r_i}>1$, and a sparsity constant $\zeta\in(0,1)$.
Let $T$ be an operator 	so that for every $f_1,\dots ,f_m,h\in C_c^\infty(\re^n)$
\begin{equation}\label{sparse-domination}
\left|\int_{\re^n} h\,T (f_1, \dots, f_m)dx\right|
\lesssim
\sup_{\mathcal{S}}\Lambda_{\mathcal{S},\vec{r}}(f_1,\dots,f_m,h),
\end{equation}
where the sup runs over all sparse families with sparsity constant $\zeta$. In \cite[Corollary 2.15]{LMO} it was shown that
for all exponents $\vec q=(q_1,\dots,q_m)$, with $1<q_1,\dots,q_m<\infty$ and $\vec{r}\prec\vec{q}$, for all weights $\vec v=(v_1,\dots,v_m) \in A_{\vec q,\vec{r}}$,
and for all $f_1,\dots, f_m\in C_c^\infty(\re^n)$ one has
\begin{equation}\label{extrapol:C:CZO}
\|T(f_1,\dots, f_m)v\|_{L^{q}} \lesssim \prod_{i=1}^m \|f_i v_i\|_{L^{q_i}},
\end{equation}
where $\frac1q:=\frac1{q_1}+\dots+\frac1{q_m}$ and $v:=\prod_{i=1}^mv_i$. From this and Theorem \ref{theor:extrapol-general} we can immediately allow the $q_i$'s to be infinity (provided that $q<\infty$). Moreover, we obtain \eqref{extrapol:vv*} for all exponents $\vec q=(q_1,\dots,q_m)$, with $1<q_1,\dots,q_m\le \infty$, $\vec{r}\prec\vec{q}$, and $\frac1q:=\frac1{q_1}+\dots+\frac1{q_m}>0$;
$\vec s=(s_1,\dots,s_m)$, with $1<s_1,\dots,s_m\le \infty$, $\vec{r}\prec\vec{q}$, and $\frac1s:=\frac1{s_1}+\dots+\frac1{s_m}>0$; and for all weights $\vec v=(v_1,\dots,v_m) \in A_{\vec q,\vec{r}}$. Again, comparing with \cite[Corollary 2.15]{LMO}, here we could have that some of the $q_i$'s (but not all) and/or some of the $s_i$'s (but not all) are infinity.

A particular case of interest is that of rough bilinear singular integrals introduced by Coifman and Meyer and further studied by \cite{GHH}. As explained in \cite[Section 2.4]{LMO}, either from the weighted estimates obtained in \cite{CHS} or from the sparse domination from \cite{Barron} one easily gets an extension of \cite[Corollary 2.17]{LMO} covering the end-point cases.

Our next example is the bilinear Hilbert transform, denoted by $BH$, which can be framed within the previous class of operators that are controlled by a sparse form as above (see \cite[Theorem 2]{CPO2016}). From that, repeating the argument in \cite[Section 2.5]{LMO} but invoking Theorem \ref{theor:extrapol-general} we obtain the following estimates. Let
$\vec{r}=(r_1,r_2,r_3)$ be such that $1<r_1,r_2,r_3<\infty$ and
\begin{equation}\label{cond-adm:corol}
\frac1{\min\{r_1,2\}}+\frac1{\min\{r_2,2\}}+\frac1{\min\{r_3,2\}}<2.
\end{equation}
Let $\vec{p}=(p_1,p_2)$, $\vec{s}=(s_1,s_2)$  with $1<p_1, p_2, s_1,s_2\le\infty$ be so that  $\frac1p:=\frac1{p_1}+\frac1{p_2}>0$,  $\frac1s:=\frac1{s_1}+\frac1{s_2}>0$.
If $\vec{r}\prec\vec{p}$ and $\vec{w}= (w_1, w_2)\in A_{\vec{p},\vec{r}}$  then
\[
\|BH(f,g) w_1w_2\|_{L^p}
\lesssim
\|f w_1\|_{L^{p_1}}\,\|g w_1\|_{L^{p_2}}.
\]
Moreover, if additionally  $\vec{r}\prec\vec{s}$ then
\[
\Big\|\big\{BH(f_j,g_j) w_1 w_2\big\}_j\Big\|_{L^{p}_{\ell^{s}}}
\lesssim
\Big\|\big\{f_j w_1\big\}_j\Big\|_{L^{p_1}_{\ell^{s_1}}} \Big\|\big\{g_j w_2\big\}_j\Big\|_{L^{p_2}_{\ell^{s_2}}}.
\]
Again, here we can allow $p_1$ or $p_2$ (but not both) and/or  $s_1$ or $s_2$ (but not both) to be infinity. The same kind of argument allows us to readily get iterated vector-valued inequalities in spaces of the form $L^{p}_{\ell^{s}_{\ell^t}}$. All these should be compared with the helicoidal method developed in \cite{BM1, BM2, BM3}, on which they prove all these estimates by some delicate discretization arguments. In a nutshell once we have the estimates from \cite[Theorem 2]{CPO2016} in hand, our powerful method based on extrapolation  easily gives all the desired estimates, including the vector-valued ones, and does not require to use any further fine analysis or decomposition of the operator. This occurs because extrapolation is not something related to operators, is a property about families o functions satisfying weighted norm inequalities.

Finally, our method allows us to deal with commutators of the previous operators with BMO functions. In \cite[Section 2.6]{LMO}, the method developed in \cite{BMMST}, was further pushed to obtain \cite[Theorem 2.22]{LMO}. The latter in conjunction with Theorem \ref{theor:extrapol-general} easily yields the following extension where both in the hypotheses and the conclusions one can include the end-points:

\begin{Corollary}\label{corol:multilinear:comm}
	Let $T$ be an $m$-linear operator and let $\vec{r}=(r_1,\dots,r_{m+1})$, with $1\le r_1,\dots,r_{m+1}<\infty$. Assume that there exists $\vec p=(p_1,\dots, p_m)$, with $1\le p_1,\dots, p_m\le \infty$ and $\vec{r}\preceq_\star\vec{p}$, such that for all $\vec{w} = (w_1, \dots, w_m)\in A_{\vec{p},\vec{r}}$, we have
	\begin{equation}
	\label{vector-weight}
	\|T(f_1, f_2,\dots, f_m) w\|_{L^p} \lesssim \prod_{i=1}^m \|f_i w_i\|_{L^{p_i}},
	\end{equation}
	where $\frac1p:=\frac1{p_1}+\dots+\frac1{p_m}$ and $w:=\prod_{i=1}^m w_i$.
	
	Then, 	for all exponents $\vec q=(q_1,\dots,q_m)$, with $1< q_1,\dots, q_m\le \infty$ so that $\vec{r}\prec\vec{q}$ and $\frac1q:=\frac1{q_1}+\dots+\frac1{q_m}>0$, for all weights $\vec v=(v_1,\dots, v_m) \in A_{\vec q,\vec{r}}$, for all $\textbf{b} = (b_1, \dots, b_m) \in  {\rm BMO}^m$, and for each multi-index $\alpha$, we have
	\begin{equation}
	\label{multi-commutator-II}
	\|[T, \textbf{b}]_\alpha (f_1, f_2,\dots, f_m) v\|_{L^q}
	\lesssim
	\prod_{i=1}^m \|b_i\|^{\alpha_i}_{{\rm BMO}} \|f_i v_i\|_{L^{q_i}},
	\end{equation}
	where $v:=\prod_{i=1}^m v_i$. Moreover if  $\vec s=(s_1,\dots,s_m)$, with $1< s_1,\dots, s_m\le \infty$ so that  $\vec{r}\prec\vec{s}$ and	$\frac1s:=\frac1{s_1}+\dots+\frac1{s_m}>0$,
	then
	\begin{equation}\label{multi-comm:VV}
	\Big\|\big\{ [T, \textbf{b}]_\alpha (f_1^j, f_2^j,\dots, f_m^j)  v\big\}_j\Big\|_{L^{p}_{\ell^{s}}}
	\lesssim
	\prod_{i=1}^m \|b_i\|^{\alpha_i}_{{\rm BMO}} \Big \|\big\{f_i^j v_i\big\}_j\Big\|_{L^{p_i}_{\ell^{s_i}}}.
	\end{equation}
\end{Corollary}

\begin{proof}
	We first use \eqref{vector-weight} and Theorem \ref{theor:extrapol-general}	to see that for all exponents $\vec t=(t_1,\dots,t_m)$, with $1< t_1,\dots, t_m\le \infty$ so that $\vec{r}\prec\vec{t}$ and $\frac1t:=\frac1{t_1}+\dots+\frac1{t_m}>0$, and for all weights $\vec v=(v_1,\dots, v_m) \in A_{\vec t,\vec{r}}$ there holds
	\begin{equation}
	\label{vector-weight:cqafcr}
	\|T(f_1, f_2,\dots, f_m) v\|_{L^t} \lesssim \prod_{i=1}^m \|f_i v_i\|_{L^{t_i}},
	\end{equation}
	where $v:=\prod_{i=1}^m v_i$. In particular, this estimate is valid for a particular choice  of $\vec{t}$ which additionally satisfies $1<t_1,\dots, t_m<\infty$. This allows us to invoke \cite[Theorem 2.22]{LMO} to see that \eqref{multi-comm:VV} is valid for the same exponents as in the statement with the additionally assumption that
	$1<q_1,\dots, q_m<\infty$. Another use of Theorem \ref{theor:extrapol-general} immediately remove that restriction and also yields the vector-valued inequalities in \eqref{multi-comm:VV}. This completes the proof.
\end{proof}

As a consequence of this result and all the applications considered above we can obtain that \eqref{multi-commutator-II} and \eqref{multi-comm:VV} hold in the following scenarios:
\begin{list}{$\bullet$}{\leftmargin=.6cm\labelwidth=.6cm\itemsep=0.2cm\topsep=.1cm}
	
	\item $T$ is an $m$-linear Calder\'on-Zygmund operator, $\vec q=(q_1,\dots,q_m)$,  $\vec{s}=(s_1,\dots, s_m)$,  with $1<q_i,s_i\le \infty$, $1\le i\le m$, so that $\frac1q:=\frac1{q_1}+\dots+\frac1{q_m}>0$ and $\frac1s:=\frac1{s_1}+\dots+\frac1{s_m}>0$, and $\vec v \in A_{\vec q }$. The same applies to a bilinear rough singular integral as in \cite[Section 2.4]{LMO} in the case $m=2$.
	
	\item $T$ is any linear operators satisfying \eqref{sparse-domination}, $\vec q=(q_1,\dots,q_m)$, $\vec s=(s_1,\dots,s_m)$, with $\vec{r}\prec\vec{q}$ and $\vec{r}\prec\vec{s}$ so that $\frac1q:=\frac1{q_1}+\dots+\frac1{q_m}>0$ and $\frac1s:=\frac1{s_1}+\dots+\frac1{s_m}>0$, and $\vec v=(v_1,\dots, v_m)  \in A_{\vec q,\vec{r}}$.
\end{list}	
These are respectively extensions of \cite[Corollaries 2.25 and 2.26]{LMO} where we are now able to consider the end-point cases, and these estimates are new to the best of our knowledge. The same occurs in the case of the bilinear Hilbert transform for which we have the following extension \cite[Corollary 2.27]{LMO} which contain new estimates at the end-point cases:

\begin{Corollary}\label{corol:BH-1st:BMO}
	Assume that $\vec{r}=(r_1,r_2,r_3)$, $1<r_1,r_2,r_3<\infty$, verifies \eqref{cond-adm:corol}.
	Let $\vec{p}=(p_1,p_2)$, $\vec{s}=(s_1,s_2)$  with $1<p_1, p_2, s_1,s_2\le\infty$ be so that  $\frac1p:=\frac1{p_1}+\frac1{p_2}>0$,  $\frac1s:=\frac1{s_1}+\frac1{s_2}>0$. If $\vec{r}\prec\vec{p}$, for all weights $\vec{w}= (w_1, w_2)\in A_{\vec{p},\vec{r}}$,  for all $\textbf{b} = (b_1, b_2) \in  {\rm BMO}^2$, and for each multi-index $\alpha=(\alpha_1,\alpha_2)$ it follows that
	\[
	\|[BH, \textbf{b}]_\alpha (f,g) w_1w_2\|_{L^p}
	\lesssim
	\|b_1\|^{\alpha_1}_{{\rm BMO}} \|b_2\|^{\alpha_2}_{{\rm BMO}} \|f w_1\|_{L^{p_1}}\|g w_2\|_{L^{p_2}},
	\]
	and, if one further assumes that $\vec{r}\prec\vec{s}$,
	\[
	\Big\|\big\{ [BH, \textbf{b}]_\alpha (f_j, g_j) w_1 w_2\big\}_j\Big\|_{L^{p}_{\ell^{s}}}
	\lesssim
	\|b_1\|^{\alpha_1}_{{\rm BMO}} \|b_2\|^{\alpha_2}_{{\rm BMO}}
	\Big\|\big\{f_j w_1\big\}_j\Big\|_{L^{p_1}_{\ell^{s_1}}} \Big\|\big\{g_j w_2\big\}_j\Big\|_{L^{p_2}_{\ell^{s_2}}}.
	\]
\end{Corollary}

\section{Mixed-norm estimates and tensor products of bilinear singular integrals}\label{section:tensor}
In \cite{LMV1}  the first, third and last authors of the present paper showed that if $T$ is a bilinear bi-parameter Calder\'on-Zygmund operator that is free of full paraproducts (see \cite{LMV1} for the definitions), then
\begin{equation}\label{eq:LMVRESULT}
\| T(f_1,f_2) w\|_{L^{p}} \lesssim \| f_1w_1 \|_{L^{p_1}}\| f_2 w_2\|_{L^{p_2}}
\end{equation}
holds for all $p_1,p_2 \in (1, \infty)$ and weights $w_1$ and $w_2$ such that $w_1^{p_1} \in A_{p_1}(\R^n \times \R^m)$ and
$w_2^{p_2} \in A_{p_2}(\R^n \times \R^m)$.  Here $p$ is defined by $\frac1p=\frac1{p_1}+\frac1{p_2}$, $w=w_1w_2$ and $A_p(\R^n \times \R^m)$
is the class of bi-parameter $A_p$ weights (obtained by replacing cubes by rectangles in the usual definition).
In addition, the same estimate were obtained for all tensor products
$T=T_n \otimes T_m$, where $T_n$ and $T_m$ are bilinear one-parameter Calder\'on-Zygmund operators in $\R^n$ and $\R^m$, respectively (we recall the definition below). Tensor products $T_n \otimes T_m$ are examples
of bilinear bi-parameter Calder\'on-Zygmund operators, which are not necessarily free of full paraproducts (if \emph{both} of them are \emph{not} free of paraproducts).
Recall that initially $T_n \otimes T_m$ is defined via
$$
(T_n \otimes T_m)(f_1 \otimes f_2, g_1 \otimes g_2)(x) := T_n(f_1, g_1)(x_1)T_m(f_2, g_2)(x_2),
$$
where $f_1, g_1 \colon \R^n \to \C$, $f_2, g_2 \colon \R^m \to \C$ and $x = (x_1, x_2) \in \R^{n+m}$.

The main goal of this section is to prove Theorem \ref{thm:WeightedTensor} where we establish weighted estimates for the previous tensor products of bilinear Calder\'on-Zygmund operators with the new class of weights introduced in Definition \ref{def:LargerClass}. In turn, we can extrapolate using Theorem \ref{theor:extrapol-general} to prove Corollary \ref{cor:TensorEndpoint}. The latter result contains some mixed-norm end-point estimates that were not proved in
\cite{LMV1}. The bottom line is that with our current extrapolation result we can consider the cases in which the outer exponents are infinity even when the inner exponents are in the quasi-Banach range.

\begin{Definition}\label{def:LargerClass}
Given $\vec{p}=(p_1,p_2)$ with $1<p_1,p_2<\infty$ let $\frac1p=\frac1{p_1}+\frac1{p_2}$. Let $\vec{w}=(w_1,w_2)$ be such that
$0<w_1(x_1,\cdot),w_2(x_1,\cdot)<\infty$ a.e. in $\R^m$ for a.e.~$x_1\in\R^n$ and $0<w_1(\cdot,x_2),w_2(\cdot,x_2)<\infty$ a.e. in $\R^n$ for a.e.~$x_2\in\R^m$. We say that $\vec{w}\in \mathbb{A}_{\vec{p}, n, m}$ if $w_1(x_1, \cdot)^{p_1}\in A_{p_1}(\R^m)$ and $w_2(x_1, \cdot)^{p_2}\in A_{p_2}(\R^m)$ for a.e. $x_1\in\R^n$; $(w_1(\cdot,x_2),w_2(\cdot,x_2))\in A_{\vec{p}}(\R^n)$ for a.e. $x_2\in\re^m$; and moreover
\begin{equation*}
\esssup_{x_2 \in \R^m} [(w_1(\cdot,x_2),w_2(\cdot,x_2)]_{A_{\vec{p}}(\R^n)}<\infty
\end{equation*}
and
\begin{equation*}
\esssup_{x_1 \in \R^n}[w_1(x_1, \cdot)^{p_1}]_{A_{p_1}(\R^m)}
+\esssup_{x_1 \in \R^n}[w_2(x_1, \cdot)^{p_2}]_{A_{p_2}(\R^m)}<\infty.
\end{equation*}
Here $A_{\vec{p}}(\R^n)$ is the class of bilinear $A_{\vec{p}}$ weights in $\R^n$ as defined in Section \ref{sec:multi} and
$A_{p_i}(\R^m)$ is the class of $A_{p_i}$ weights in $\R^m$.
\end{Definition}

A function $K \colon \R^{3d} \setminus \Delta \to \C$, where $\Delta=\{(x,x,x) \in \R^{3d}:\, x \in \R^d\}$,
is called a bilinear singular integral kernel if for some $\alpha \in (0,1]$ and a constant $C>0$
it satisfies the estimates
\[
|K(x,y,z)| \le C \frac{1}{(|x-y|+|x-z|)^{2d}}
\]
for  $(x,y,z) \in \R^{3d} \setminus \Delta$,
\begin{multline*}
|K(x',y,z)-K(x,y,z)|
+
|K(y,x',z)-K(y,x,z)|
\\
+
|K(y,z,x')-K(y,z,x)|
\le C \frac{|x-x'|^\alpha}{(|x-y|+|x-z|)^{2d+\alpha}}
\end{multline*}
for $(x',y,z),(x,y,z) \in \R^{3d} \setminus \Delta$ such that
$|x'-x| \le \max (|x-y|,|x-z|)/2$. We denote the smallest possible constant
$C$ in these inequalities by $\| K \|_{\operatorname{CZ}_\alpha}$.

We say that $T$ is a bilinear singular integral in $\R^d$ if there exists a bilinear singular integral kernel $K$ such that
\begin{equation}\label{eq:KernelRep}
\langle T(f_1,f_2),f_3 \rangle
= \iiint_{\R^{3d}} K(x,y,z) f_1(y)f_2(z)f_3(x) d x d y d z
\end{equation}
for all $f_1,f_2,f_3\in L^\infty_c(\re^d)$ whenever two of the functions are disjointly supported, that is,  $\supp f_i \cap \supp f_j=\emptyset$ for some $i\neq j$.
We say that $T$ is a bilinear Calder\'on-Zygmund operator in $\R^d$ if $T$ is a bilinear singular integral in $\R^d$ and moreover $T$ is bounded
from $L^{p_1}(\R^d) \times L^{p_2}(\R^d)$ into $L^p(\R^d)$ for some $1<p_1,p_2\le \infty$ such that $\frac1p=\frac1{p_1}+\frac1{p_2}>0$.
For more information about multilinear Calder\'on-Zygmund operators see \cite{GT} where it is shown, among other things, that
$T$ is bounded from $L^{q_1}(\R^d) \times L^{q_2}(\R^d)$ into $L^q(\R^d)$ for all $q_1,q_2 \in (1, \infty]$ and
$\frac1q=\frac1{q_1}+\frac1{q_2}>0$.

So-called $T1$ theorems give conditions under which singular integrals are indeed Calderón-Zygmund operators.  For bilinear $T1$ theorems see for example \cite{GT} (where multilinear singular integrals are also treated) or \cite{LMOV}.

The following theorem contains our weighted estimates for tensor products of bilinear Calder\'on-Zygmund operators involving the class of weights appearing in Definition \ref{def:LargerClass}.

\begin{Theorem}\label{thm:WeightedTensor}
Let $T_n$ and $T_m$ be bilinear Calder\'on-Zygmund operators in $\R^n$ and $\R^m$, respectively.
Let $\vec{p}=(p_1,p_2)$ be such that $1<p_1, p_2<\infty$ and set $\frac1p=\frac1{p_1}+\frac1{p_2}$.
Then
$$
\| T_n \otimes T_m (f_1,f_2) w_1w_2\|_{L^p}
\lesssim \| f_1 w_1\|_{L^{p_1}}\| f_2w_2 \|_{L^{p_2}}
$$
holds for all weights $\vec{w}=(w_1,w_2) \in \mathbb{A}_{\vec{p}, n, m}$  and for all $f_1w_1\in L^{p_1}(\R^{n+m})$ and $f_2w_2\in L^{p_2}(\R^{n+m})$.
\end{Theorem}

Given $0<p,q\le \infty $,  the space $L^{p}(\R^n; L^{q}(\R^m))$ consists of measurable functions $f \colon \R^{n+m} \to \C$
such that
\[
\|f\|_{L^{p}(\R^n; L^{q}(\R^m))}
:=
\big\| \|f(x_1, x_2)\|_{L_{x_2}^q(\R^m)} \big\|_{L_{x_1}^p(\R^n)} < \infty.
\]

Combining Theorem \ref{thm:WeightedTensor} with our main extrapolation result, Theorem \ref{theor:extrapol-general}, we will easily get the following consequence:

\begin{Corollary}\label{cor:TensorEndpoint}
	Let $T_n$ and $T_m$ be bilinear Calder\'on-Zygmund operators in $\R^n$ and $\R^m$, respectively.
	Let $1<p_1,p_2\le \infty$ and $1<q_1,q_2<\infty$ be such that
	$\frac1p=\frac1{p_1}+\frac1{p_2}>0$.
	Then
	\begin{equation}\label{qfwafaerfer}
	\| T_n \otimes T_m (f_1,f_2) \|_{L^{p}(\R^n; L^{q}(\R^m))}
	\lesssim \|f_1 \|_{L^{p_1}(\R^n; L^{q_1}(\R^m))}
	\| f_2 \|_{L^{p_2}(\R^n; L^{q_2}(\R^m))}.
	\end{equation}
	 for all $f_1, f_2\in {L_c^{\infty}(\R^{n+m})}$.
\end{Corollary}

Our last goal is to extend the previous result in the following ways. First, we would like to consider the end-points  $q_1=\infty$ or $q_2=\infty$. Also, we wish  to show that \eqref{qfwafaerfer} holds for all  $f_1\in {L^{p_1}(\R^n; L^{q_1}(\R^m))}$ and $f_2\in {L^{p_2}(\R^n; L^{q_2}(\R^m))}$. This is straightforward if $p_1,p_2,q_1,q_2<\infty$ by a standard density argument. However, such approach fails when some of the exponents is infinity and in that scenario one even needs to make sense of $T_n \otimes T_m (f_1,f_2)$. All this is done in the main result of this section:

\begin{Theorem}\label{theor:TensorEndpoint}
Let $T_n$ and $T_m$ be bilinear Calder\'on-Zygmund operators in $\R^n$ and $\R^m$, respectively.
Let $1<p_1,p_2, q_1,q_2\le \infty$ be such that
$\frac1p=\frac1{p_1}+\frac1{p_2}>0$ and $\frac1q=\frac1{q_1}+\frac1{q_2}>0$.
Then we have
\begin{equation}\label{qfwafaerfer:theor}
\| T_n \otimes T_m (f_1,f_2) \|_{L^{p}(\R^n; L^{q}(\R^m))}
\lesssim \|f_1 \|_{L^{p_1}(\R^n; L^{q_1}(\R^m))}
\| f_2 \|_{L^{p_2}(\R^n; L^{q_2}(\R^m))}
\end{equation}
for all $f_1\in {L^{p_1}(\R^n; L^{q_1}(\R^m))}$ and $f_2\in {L^{p_2}(\R^n; L^{q_2}(\R^m))}$.
\end{Theorem}

To prove the previous results we need some preliminaries. Let $\Omega_d:= (\{0,1\}^d)^\Z$ be the set of $\{0,1\}^d$-valued sequences $(\omega_i)_{i \in \Z}$ equipped with the probability measure  such that the coordinates are independent and uniformly distributed over $\{0,1\}^d$.   We write $\mathbb{E}_\omega$ to denote the associated conditional expectation.

Let
$$
\calD_0^d:=\{2^{-k}([0,1)^d+m):\ k \in \Z,\ m \in \Z^d\}
$$
be the standard dyadic lattice of cubes in $\R^d$. For $\omega \in \Omega_d$ define the dyadic lattice
$$
\calD_\omega^d:= \{ I+\omega \colon I \in \calD_0^d\},
\qquad I+\omega := I+ \sum_{i \colon 2^{-i} < \ell(I)} 2^{-i} \omega_i.
$$
By a dyadic lattice $\calD^d$ we mean that $\calD^d=\calD^d_\omega$ for some $\omega\in \Omega_d$.

Let $\calD^d$ be a dyadic lattice and suppose $I \in \calD^d$. The sidelength of $I$  is denoted by $\ell(I)$.
If $k =0,1,2, \dots$, then $I^{(k)}$ denotes the $k$-ancestor, that is, $I^{(k)}$ is the unique cube in $\calD^d$ such that $I \subset I^{(k)}$ and
$\ell(I^{(k)})=2^k \ell(I)$.  We define $\ch(I)$ to be the collection of dyadic children of $I$, that is, those $I' \in \calD^d$ such that $(I')^{(1)}=I$.

Suppose $f \in L^1_{\loc}(\R^d)$ and $I \in \calD^d$.
The average $\langle f \rangle_I$ and the martingale difference $\Delta_I f$ are defined by
\[
\langle f \rangle_I:= \frac{1}{|I|} \int_I f dx, \qquad
 \Delta_I f:= \sum_{I' \in \ch (I)} \langle f \rangle_{I'}1_{I'}-\langle f \rangle_I1_I.
\]

Given $I \in \calD^d$ we denote by $h_I$ a cancellative $L^2$ normalized Haar function. That is, writing $I = I_1 \times \cdots \times I_d$ we can define the Haar function $h_I^{\eta}$, $\eta = (\eta_1, \ldots, \eta_d) \in \{0,1\}^d$, by setting
$
h_I^{\eta} = h_{I_1}^{\eta_1} \otimes \cdots \otimes h_{I_d}^{\eta_d},
$
where $h_{I_i}^0 = |I_i|^{-\frac12}1_{I_i}$ and $h_{I_i}^1 = |I_i|^{-\frac12}(1_{I_{i, l}} - 1_{I_{i, r}})$ for every $i = 1, \ldots, d$. Here $I_{i,l}$ and $I_{i,r}$ are the left and right
halves of the interval $I_i$ respectively. If $\eta \in \{0,1\}^d \setminus \{0\}^d$ the Haar function is cancellative, that is, $\int h_I^{\eta} dx= 0$. We usually suppress $\eta$ and simply write $h_I$ to mean that $h_I^{\eta}$, for some $\eta \in \{0,1\}^d \setminus \{0\}^d$ and we write $h_I^0$ to denote $h_I^{\eta}$ for $\eta=(0,\dots,0)$.  It is standard to see that for every $I\in\calD^d$ one has
\begin{equation}\label{Haar-vs-Martingale}
\Delta_I f=\sum_{\eta\in\{0,1\}^d\setminus\{0\}^d} \langle f, h_I^\eta\rangle h_I^\eta.
\end{equation}

Let $\calD=\calD^d$ be a dyadic lattice in $\R^d$ and let $k=(k_1,k_2,k_3)\in\N_0^3$ ---here and elsewhere $\N_0=\N\cup\{0\}$.
A bilinear dyadic shift $\mathcal{U}^k_{\calD}$ of complexity $k$ is an operator that
has three possible different forms. One of them is
\begin{equation}\label{dyadic-shift}
\mathcal{U}^k_\calD(f_1,f_2)
=\sum_{K \in \calD} \sum_{\substack{I_1, I_2, I_3 \in \calD \\ I_i^{(k_i)}=K}}
a_{K,(I_i)} \langle f_1,  h^0_{I_1}\rangle \langle f_2,  h_{I_2}\rangle h_{I_3},
\end{equation}
where every $a_{K,(I_i)}=a_{K,I_1,I_2,I_3}$ satisfies
$$
|a_{K,(I_i)}| \le \frac{|I_1|^{\frac12} |I_2|^{\frac12} |I_3|^{\frac12}}{|K|^2}.
$$
Another form of the shift is formed by replacing the Haar functions in the above formula by $h_{I_1}$, $h_{I_2}^0$
and $h_{I_3}$, and in the third form we have $h_{I_1}$, $h_{I_2}$ and $h_{I_3}^0$.

A bilinear dyadic paraproduct $\mathcal{U}_\calD$ has also three possible forms. One of them is
\begin{equation}\label{dyadic-para}
\mathcal{U}_\calD(f_1,f_2)
=\sum_{K \in \calD} a_K \langle f_1 \rangle_K\langle f_2 \rangle_K h_K,
\end{equation}
where the coefficients satisfy
$$
\sup_{K_0 \in \calD} \Big( \frac{1}{|K_0|} \sum_{\substack{K \in \calD \\ K \subset K_0}} |a_K|^2 \Big)^{\frac12}
\le 1.
$$
The second form is $ \langle f_1, h_K \rangle \langle f_2 \rangle_K \frac{1_K}{|K|}$ and third one is $ \langle f_1 \rangle_K  \langle f_2, h_K \rangle \frac{1_K}{|K|}$.

We will refer to bilinear dyadic shifts and paraproducts as bilinear dyadic   model operators.

Suppose $T$ is a bilinear Calder\'on-Zygmund operator in $\R^d$ related to a kernel $K$. We will repeatedly  use the following bilinear one-parameter representation theorem from \cite{LMOV}: if $f_1, f_2, f_3\in L^3(\R^d)$ then
\begin{equation}\label{eq:repthm}
\langle T(f_1,f_2), f_3 \rangle
=C_T \E_\omega \sum_{k\in \N_0^3}  \sum_{u=1}^{C_d} 2^{-\max_i k_i \frac{\alpha}{2}}
\langle U^k_{\calD_\omega,u} (f_1,f_2),f_3\rangle.
\end{equation}
Here $|C_T| \lesssim \| T \|_{L^3 \times L^3 \to L^{3/2}}+ \| K \|_{\operatorname{CZ}_\alpha}$ and
$\alpha$ is the parameter in the H\"older continuity assumptions of the kernel $K$ of $T$. Also, $C_d$ is a finite constant depending just on $d$.  For each $u$ and $k=(k_1,k_2,k_3)\in\N_0^3$, if $\max {k_i}>0$, then $U^k_{\calD_\omega,u}$ is a bilinear dyadic shift of complexity $k$ with respect to the lattice $\calD_{\omega}$,
that is, it has the form $\mathcal{U}^k_{\calD_\omega}$ (see \eqref{dyadic-shift}). On the other hand, if $\max k_i=0$, then $U^k_{\calD_\omega,u}$ is either a bilinear dyadic shift of complexity $(0,0,0)$ of the form $\mathcal{U}^{(0,0,0)}_{\calD_\omega}$ (see \eqref{dyadic-shift}), or a bilinear dyadic paraproduct of the form $\mathcal{U}_{\calD_\omega}$ (see \eqref{dyadic-para}).

Let $\vec{p}=(p_1,p_2)$ with $1<p_1, p_2\le \infty$ such that $1/p=1/p_1+1/p_2 > 0$.  Given $\vec{k}=(k_1,k_2,k_3)$, let $U^k_{\calD}$ be a bilinear dyadic shift of complexity $k$ if
$\max {k_i}>0$, or if $\max k_i=0$ assume that $U^k_\calD$ is either a bilinear dyadic shift of complexity $(0,0,0)$ or a bilinear  dyadic paraproduct. We claim that  for $f_1, f_2 \in L^{\infty}_c(\R^d)$ we have
\begin{equation}\label{eq:ModelWeighted}
\| U^k_\calD (f_1,f_2) w\|_{L^p}
\lesssim   (\max_i k_i+1) \| f_1w_1 \|_{L^{p_1} } \| f_2 w_2\|_{L^{p_2}},
\end{equation}
where $\vec{w}=(w_1,w_2) \in A_{\vec{p}}(\R^d)$, $w=w_1w_2$ and the implicit constant depends on the $A_{(p_1, p_2)}$ characteristic of $(w_1,w_2) $.
To obtain these estimates we first observe that in \cite[Section 5]{LMOV} it was shown that bilinear dyadic shifts and paraproducts satisfy the so-called sparse form domination. Second, in \cite[Theorem 3.2]{LMS} it was shown that sparse operators satisfy \eqref{eq:ModelWeighted}
under the further assumption  $p_1,p_2<\infty$. Finally we can apply Theorem \ref{theor:extrapol-general:CZO} to conclude as well the case on which either $p_1$ or $p_2$ are infinity.

\begin{Lemma}\label{lem:LowerSF}
Let $\Omega_0$ be a probability space and let $\{\mathcal{D}^d_\omega\}_{\omega\in\Omega_0}$ be a collection of dyadic grids in $\R^d$.
Given $0<p<\infty$ and $\nu \in A_\infty(\R^d)$, for every sequence of functions $\{g_\omega\}_{\omega\in \Omega_0}$ with $g_\omega \colon \R^d \to \C$ for every $\omega \in \Omega_0$ one has
$$
\| \E_\omega g_{\omega} \|_{L^p(\nu)}
\lesssim  \Big\| \E_\omega \Big( \sum_{J \in \calD^d_{\omega}} |\Delta_J g_\omega |^2 \Big)^{\frac12} \Big \|_{L^p(\nu)}.
$$
\end{Lemma}

\begin{proof}
Given $\calD^d$, a dyadic lattice in $\R^d$, one can first show that
\begin{equation}\label{awFQFAWf}
\| g \|_{L^{t}(\nu)} \lesssim \Big \| \Big( \sum_{J \in \calD^d} |\Delta_J g |^2 \Big)^{\frac12} \Big \|_{L^{t}(\nu)}
\end{equation}
for any $t \in (0, \infty)$ and  $\nu \in A_{\infty}(\R^d)$ where the implicit constant is independent of $\calD^d$. To see this we can first invoke \cite{CWW} or \cite[Theorem 3.4]{DIPTV} to obtain the case $t=2$. In turn, using $A_\infty$-extrapolation, see \cite[Theorem 2.1]{CUMP}, we conclude that the same estimate holds for all $0<t<\infty$.

We next use \eqref{awFQFAWf} with $t=1$ to see that for every $\nu \in A_\infty(\R^m)$ we have
\begin{multline*}
\| \E_\omega g_{\omega} \|_{L^1(\nu)}
\le \E_\omega \|  g_{\omega} \|_{L^1(\nu)}
\lesssim
\E_\omega \Big\| \Big( \sum_{J \in \calD^d_{\omega}} |\Delta_J g_\omega |^2 \Big)^{\frac12} \Big \|_{L^1(\nu)}
\\
=
\Big\| \E_\omega \Big( \sum_{J \in \calD^d_{\omega}} |\Delta_J g_\omega |^2 \Big)^{\frac12} \Big \|_{L^1(\nu)}.
\end{multline*}
We use again $A_\infty$ extrapolation, see \cite[Theorem 2.1]{CUMP},  to obtain the desired estimate.
\end{proof}

Given a sequence of scalars $\{a_{i,j}\}_{i,j=1}^\infty\subset\mathbb{C}$ and $p,q \in (0, \infty]$, we define
$$
\big\| \{a_{i,j}\}_{i,j} \big\|_{\ell_j^p(\ell_i^q)}:= \Big \| \big\{ \big\| \{a_{i,j}\}_i \big\|_{\ell^q} \big\}_j \Big \|_{\ell^p}.
$$
The following lemma states a variant of the bilinear Marcinkiewicz-Zygmund inequality:

\begin{Lemma}\label{lem:MZ}
Let $\Omega_0$ be a probability space. Assume that the sequence of bilinear operators $\{T^\omega\}_{\omega\in \Omega_0}$ satisfies
\begin{equation}\label{f4q3r34tfg3g}
\sup_{\omega\in\Omega_0}\| T^\omega(f_1,f_2)w_1 w_2 \|_{L^{3/2}} \lesssim  \| f_1w_1 \|_{L^3} \| f_2 w_2\|_{L^3}
\end{equation}
for all weights $\vec{w}=(w_1,w_2)\in A_{(3,3)}(\R^d)$ and $f_1, f_2\in L^\infty_c(\R^d)$.
Let $1<s\le 2$, $\vec{p}=(p_1,p_2)$, and $\vec{q}=(q_1,q_2)$ be such that $1<p_1,p_2,q_1,q_2\le \infty$, $\frac1p=\frac1{p_1}+\frac1{p_2}>0$ and $\frac1q=\frac1{q_1}+\frac1{q_2}>0$. Then,
the estimate
\begin{multline*}
\Big\| \E_{\omega}  \Big( \sum_{k=1}^\infty \Big( \sum_{i,j=1}^\infty |T^\omega(f^\omega_{1,i,k}, f^\omega_{2,j,k})|^s \Big)^{\frac{q}{s}} \Big)^{\frac{1}{q}} w_1w_2\Big\|_{L^{p}}
\\
\lesssim
\Big \|  \Big(\E_\omega \big\| \{f^\omega_{1,i,k}\}_{i,k} \big\|_{\ell_k^{q_1}(\ell_i^s)}^2 \Big)^{\frac{1}{2}} w_1\Big \|_{L^{p_1}}
\Big \|  \Big(\E_\omega \big\| \{f^\omega_{2,j,k}\}_{j,k} \big\|_{\ell_k^{q_2}(\ell_j^s)}^2 \Big)^{\frac{1}{2}} w_2\Big \|_{L^{p_2}}
\end{multline*}
holds for all weights $\vec{w}=(w_1,w_2)\in A_{\vec{p}}(\R^d)$ and for $f_{1,i,k}^\omega, f_{2,j,k}^\omega\in L^\infty_c(\R^d)$.
\end{Lemma}

\begin{proof}
Take an arbitrary $\omega \in \Omega_0$. Let $\vec{r}=(r_1,r_2)$ be such that $1<r_1,r_2<s$ and set $\frac1r=\frac1{r_1}+\frac1{r_2}$.  Note that in particular $\frac1r>\frac2s\ge 1$. On the other hand, from \eqref{f4q3r34tfg3g} Theorem \ref{theor:extrapol-general:CZO} we deduce that $T^\omega$
satisfies the estimate
$$
\| T^\omega(f_1,f_2)w_1w_2\|_{L^r}
\lesssim \| f_1 w_1\|_{L^{r_1}}\| f_2 w_2\|_{L^{r_2}}.
$$
for all $\vec{w}=(w_1,w_2)\in A_{\vec{r}}(\R^d)$, and $f_1, f_2\in L^\infty_c(\R^d)$.
With this in hand, since $0< r<\max\{r_1,r_2\}<s\le 2$ we can apply the bilinear Marcinkiewicz-Zygmund inequality in \cite[Proposition 5.3]{CMO} to obtain that
\begin{equation}\label{eq:ApplyCMO}
\Big \| \Big(\sum_{i,j=1}^\infty |T^\omega(f_{1,i},f_{2,j})|^s \Big)^{\frac{1}{s}} w_1w_2\Big \|_{L^r}
\lesssim
\Big \| \big\| \{f_{1,i}\}_i\big\|_{\ell^s} w_1\Big\|_{L^{r_1}}\Big\| \big\|\{f_{2,j}\}_j \big\|_{\ell^s} w_2\Big\|_{L^{r_2}},
\end{equation}
for all $\vec{w}=(w_1,w_2)\in A_{\vec{r}}(\R^d)$ and $f_{1,i}, f_{2,j}\in L^\infty_c(\R^d)$.

By Theorem \ref{theor:extrapol-general:CZO}, for every $\vec{q}=(q_1,q_2)$, $\vec{t}=(t_1,t_2)$ with $1<q_1,q_2,t_1,t_2\le \infty$ so that
$\frac1q=\frac1{q_1}+\frac1{q_2}>0$ and $\frac1t=\frac1{t_1}+\frac1{t_2}>0$ one can immediately obtain that
\begin{multline*}
\Big\|   \Big( \sum_{k=1}^\infty \Big( \sum_{i,j=1}^\infty |T^\omega(f_{1,i,k}, f_{2,j,k})|^s \Big)^{\frac{q}{s}} \Big)^{\frac{1}{q}} w_1 w_2\Big\|_{L^{t}} \\
 \lesssim \Big \|  \big\| \{f_{1,i,k}\}_{i,k} \big\|_{\ell_k^{q_1}(\ell_i^s)} w_1\Big \|_{L^{t_1}}
\Big \|  \big\| \{f_{2,j,k}\}_{j,k} \big\|_{\ell_k^{q_2}(\ell_j^s)} w_2 \Big \|_{L^{t_2}}.
\end{multline*}
holds for every $\vec{w}=(w_1,w_2)\in A_{\vec{t}}(\R^d)$  and $f_{1,i,k}, f_{2,j,k}\in L^\infty_c(\R^d)$. Here it is crucial to emphasize all the implicit constants are uniformly bounded on $\omega\in \Omega_0$ because of \eqref{f4q3r34tfg3g}. Applying this for the particular choice $\vec{t}=(2,2)$ so that $t=1$ one can easily see that
\begin{equation*}
\begin{split}
\Big\| \E_\omega  \Big( \sum_{k=1}^\infty & \Big( \sum_{i,j=1}^\infty |T^\omega(f_{1,i,k}, f_{2,j,k})|^s \Big)^{\frac{q}{s}} \Big)^{\frac{1}{q}} w_1w_2\Big\|_{L^{1}} \\
&
\le \E_\omega  \Big\| \Big( \sum_{k=1}^\infty\Big( \sum_{i,j=1}^\infty |T^\omega(f_{1,i,k}, f_{2,j,k})|^s \Big)^{\frac{q}{s}} \Big)^{\frac{1}{q}} w_1w_2\Big\|_{L^{1}} \\
&
\lesssim
\E_\omega \Big \|  \big\| \{f^\omega_{1,i,k}\}_{i,k} \big\|_{\ell_k^{q_1}(\ell_i^s)} w_1\Big \|_{L^{2}}
\Big \|  \big\| \{f^\omega_{2,j,k}\}_{j,k} \big\|_{\ell_k^{q_2}(\ell_j^s)} w_2 \Big \|_{L^{2}} \\
 & \le \Big \| \Big( \E_\omega \big\| \{f^\omega_{1,i,k}\}_{i,k} \big\|_{\ell_k^{q_1}(\ell_i^s)}^2 \Big)^{\frac{1}{2}} w_1\Big \|_{L^{2}}
\Big \| \Big( \E_\omega \big\| \{f_{2,j,k}\}_{j,k} \big\|_{\ell_k^{q_2}(\ell_j^s)}^2 \Big)^{\frac{1}{2}} w_2\Big \|_{L^{2}}
\end{split}
\end{equation*}
for every $\vec{w}=(w_1,w_2)\in A_{\vec{t}}(\R^d)$  and $f_{1,i,k}, f_{2,j,k}\in L^\infty_c(\R^d)$. We extrapolate from this and Theorem \ref{theor:extrapol-general:CZO} readily leads to the desired estimate.
\end{proof}

\begin{Lemma}\label{lem:block-SF}
Let $\Omega_0$ be a probability space and let $\{\mathcal{D}^d_\omega\}_{\omega\in\Omega_0}$ be a collection of dyadic grids in $\R^d$. For every $\omega\in\Omega_0$,  $I\in\mathcal{D}^d_\omega$ and $k\ge 0$ set
\[
\Delta_{I,k} f:=\sum_{\substack{J\in\mathcal{D}^d_\omega\\ J^{(k)}=I}} \Delta_J f.
\]
If $1<p<\infty$ and $\nu\in A_{p}(\R^d)$ then
\[
\sup_{k\ge 0}
\Big \| \Big( \E_{\omega} \sum_{I \in \calD^d_{\omega}}
( M (\Delta_{I,k} g ))^2  \Big)^{\frac 12} \Big \|_{L^{p}(\nu)}
\lesssim
\| g \|_{L^{p}(\nu)}.
\]
\end{Lemma}

\begin{proof}
Fix $\omega\in\Omega_0$ and note that clearly
\[
\sum_{I \in \calD^d_{\omega}} |\Delta_{I,k} g|^2
=
\sum_{I \in \calD^d_{\omega}}
\sum_{\substack{J \in \calD^d_{\omega}\\J^{(k)}=I}} |\Delta_{J} g|^2
=
\sum_{I \in \calD^d_{\omega}}
|\Delta_{I} g|^2.
\]
Hence, for every $\nu\in A_2(\R^d)$, the boundedness of $M$ and the dyadic square function on $L^2(\nu)$  imply
\begin{multline*}
\sup_k\Big \| \Big( \E_{\omega} \sum_{I \in \calD^d_{\omega}}
( M (\Delta_{I,k} g ))^2  \Big)^{\frac 12} \Big \|_{L^{2}(\nu)}^2
=
\sup_k
\E_{\omega} \sum_{I \in \calD^d_{\omega}}\int_{\R^d} ( M (\Delta_{I,k} g ))^2 \nu dx
\\
\lesssim
\sup_k\E_{\omega} \sum_{I \in \calD^d_{\omega}}\int_{\R^d} |\Delta_{I,k} g|^2 \nu dx
=\E_{\omega}\Big \| \Big( \sum_{I \in \calD^d_{\omega}} |\Delta_{I} g|^2  \Big)^{\frac 12} \Big \|_{L^{2}(\nu)}^2
\lesssim
\|g\|_{L^2(\nu)}^2.
\end{multline*}
We can now invoke Rubio de Francia extrapolation theorem to obtain at once the desired inequality.
\end{proof}

\begin{Lemma}\label{lemma:paraprod}
Let $\Omega_0$ be a probability space and let $\{\mathcal{D}^d_\omega\}_{\omega\in\Omega_0}$ be a collection of dyadic grids in $\R^d$. Suppose we have a sequence $\{\gamma_I^{\omega}\}_{\omega\in\Omega_0,I\in\mathcal{D}^d_\omega}$ of complex numbers such that
\[
\sup_{\omega\in\Omega_0}\Big(\sup_{I_0\in\mathcal{D}^d_\omega} \frac1{|I_0|} \sum_{I\in\mathcal{D}^d_\omega, I\subset I_0} |\gamma_I^{\omega}|^2\Big)^{\frac12}\le 1.
\]
Then, for every $1<p<\infty$ and $\nu\in A_p(\R^d)$ we have
\[
\Big \| \Big(\E_{\omega} \sum_{I \in \calD^d_{\omega}} |\gamma_I^{\omega} \langle g \rangle_I|^2\,\frac{1_I}{|I|} \Big)^{\frac 12} \Big\|_{L^p(\nu)}
\lesssim
\|g\|_{L^p(\nu)}.
\]
\end{Lemma}

\begin{proof}
By Rubio de Francia extrapolation we just need to obtain the case $p=2$. Assuming that $\nu\in A_2(\R^d)$ one can see that if we write $f_\omega=\sum_{I \in \calD^d_{\omega}} \gamma_I^{\omega} \langle g \rangle_I\,h_I^\eta$ (for some fixed $\eta \in \{0,1\}^d \setminus \{0\}^d$) then \eqref{Haar-vs-Martingale} easily yields
\begin{multline*}
\Big \| \Big(\E_{\omega} \sum_{I \in \calD^d_{\omega}} |\gamma_I^{\omega} \langle g \rangle_I|^2 \frac{1_I}{|I|}\Big)^{\frac 12} \Big\|_{L^2(\nu)}^2
=
\E_{\omega}  \Big \| \Big(\sum_{I \in \calD^d_{\omega}} |\Delta_I f_\omega|^2\Big)^{\frac 12} \Big\|_{L^2(\nu)}^2
\lesssim
\E_{\omega}  \|f_\omega\|_{L^2(\nu)}^2
\\
=
\E_{\omega}  \Big \| \sum_{I \in \calD^d_{\omega}}\gamma_I^{\omega} \langle g \rangle_I\,h_I^\eta\Big\|_{L^2(\nu)}^2
\lesssim
\|g\|_{L^2(\nu)}^2
,
\end{multline*}
where we have used the $L^2(\nu)$-boundedness of the dyadic square function and the dyadic paraproduct (with the dyadic-BMO function $b_\omega=\sum_{I \in \calD^d_{\omega}}\gamma_I^{\omega} h_I^\eta$), see for instance \cite{CMP1}.
\end{proof}

\subsection{Proof of Theorem \ref{thm:WeightedTensor}}

To set the stage we fix $\vec{p}=(p_1,p_2)$ such that $1<p_1, p_2<\infty$ and set $\frac1p=\frac1{p_1}+\frac1{p_2}$. We also pick $\vec{w}=(w_1,w_2) \in \mathbb{A}_{\vec{p}, n, m}$.

Let $f_1, f_2, f_3$ be finite linear combinations of tensor products of bounded and compactly supported functions. We apply the representation theorem from \cite{LMOV} (see \eqref{eq:repthm}) to both $T_n$ and $T_m$. For $T_n$ we will use $\omega_1$  to denote the random parameter and $U^k_{\omega_1,u_1}$ for the different bilinear model operators. Analogously, for $T_m$ these will be respectively $\omega_2$ and $U^v_{\omega_2,u_2}$.  We also let $\E_\omega:= \E_{\omega_1}\E_{\omega_2}$. Using this notation \eqref{eq:repthm} readily leads to
\[
\langle T_n \otimes T_m(f_1, f_2), f_3 \rangle = C_{T_n}C_{T_m}
 \E_\omega \sum_{k \in \N_0^3} \sum_{v \in \N_0^3}   \sum_{u_1=1}^{C_n}\sum_{u_2=1}^{C_m}
c_{k,v} \langle U^k_{\omega_1,u_1} \otimes U^v_{\omega_2,u_2} (f_1,f_2), f_3 \rangle,
\]
where $c_{k,v}:= 2^{-\max_i k_i \frac{\alpha_1}2 -\max_i v_i \frac{\alpha_2}2}$ and $\alpha_1$ and $\alpha_2$ are respectively the  parameters in the H\"older continuity assumptions of the kernels of $T_n$ and $T_m$. It is not too hard to show, and this also follows directly from \cite{LMV1}, that
$$
|\langle U^k_{\omega_1,u_1} \otimes U^v_{\omega_2,u_2} (f_1,f_2), f_3 \rangle| \lesssim \|f_1\|_{L^3} \|f_2\|_{L^3} \|f_3\|_{L^3},
$$
 uniformly in $\omega_1, \omega_2, k,v$.
Since finite linear combinations of tensor products of bounded and compactly supported functions are dense in $L^3$, this implies
via the above representation that $T_n \otimes T_m \colon L^3 \times L^3 \to L^{\frac32}$ is bounded, and that
the representation holds for $f_1,f_2,f_3 \in L^\infty_c(\R^{n+m})$.

Take next $f_1, f_2 \in L^{\infty}_c(\R^{n+m})$. It is easy to see that if $p \ge 1$ we have
$$
\| T_n \otimes T_m (f_1,f_2)w\|_{L^p }
\lesssim \sum_{k \in \N_0^3}\sum_{v \in \N_0^3} \sum_{u_1,u_2}
c_{k,v}
\| \E_\omega  U^k_{\omega_1,u_1} \otimes U^v_{\omega_2,u_2} (f_1,f_2) w\|_{L^p},
$$
and if $p<1$, then
$$
\| T_n \otimes T_m(f_1,f_2) w\|_{L^p }^p
\lesssim \sum_{k \in \N_0^3}\sum_{v \in \N_0^3}  \sum_{u_1,u_2}
c_{k,v}^p
\| \E_\omega  U^k_{\omega_1,u_1} \otimes U^v_{\omega_2,u_2} (f_1,f_2)w \|_{L^p}^p.
$$
We claim that
\begin{equation}\label{eq:TensModWeighted}
\| \E_\omega  U^k_{\omega_1,u_1} \otimes U^v_{\omega_2,u_2} (f_1,f_2) w\|_{L^p}
\lesssim
C_{k,v}\| f_1 w_1\|_{L^{p_1}}\| f_2w_2 \|_{L^{p_2}}
\end{equation}
holds with $C_{k,v}:=2^{\max_i k_i \frac{\alpha_1}4 +\max_i v_i \frac{\alpha_2}4}$. With this in hand, all the previous observations readily lead us to the desired estimate for $f_1, f_2 \in L^{\infty}_c(\R^{n+m})$. By density this completes the proof of Theorem \ref{thm:WeightedTensor}.

We are left with showing our claim \eqref{eq:TensModWeighted}. Let us observe that when $p \ge1$ we could take the expectation
out of the $L^p$ norm and prove a uniform estimate for all $\omega$. However, in the case $p<1$ we cannot do this, and therefore we need to keep working with the expectation. We consider two cases, whether
$U^v_{\omega_2,u_2}$ is a bilinear dyadic shift or is a bilinear dyadic paraproduct.

\medskip

\noindent\textbf{Case 1:} Proof of \eqref{eq:TensModWeighted} when $U^v_{\omega_2,u_2}$ is a dyadic bilinear shift.

To simplify the notation we write $U^k_{\omega_1}=U^k_{\omega_1,u_1}$ and $U^v_{\omega_2}=U^v_{\omega_2,u_2}$. In this scenario $U^v_{\omega_2}$ is of the form $\mathcal{U}_{\mathcal{D}_{\omega_2}^m}^v$ (see \eqref{dyadic-shift}). Recall that dyadic shifts have three different expressions depending where the non-cancellative Haar function is located. For now we consider the case on which  the non-cancellative Haar function is in the ``second'' position, that is,
\[
U^v_{\omega_2}(g_1,g_2)
=
\sum_{V \in \calD^m_{\omega_2}} \sum_{\substack{J_1, J_2,J_3 \in \calD^m_{\omega_2} \\ J_i^{(v_i)}=V}}
a^{\omega_2}_{V, (J_i)}\langle g_1, h_{J_1} \rangle, \langle g_2, h_{J_2}^0 \rangle h_{J_3}.
\]
The other possibilities of $U^v_{\omega_2}$ will be handled in a similar manner.

If we next write $U_\omega:= U^k_{\omega_1} \otimes U^v_{\omega_2}$ one can easily see that
\begin{equation}\label{eq:SpecialShift}
U_{\omega}(f_1,f_2)
=\sum_{V \in \calD^m_{\omega_2}} \sum_{\substack{J_1, J_2,J_3 \in \calD^m_{\omega_2} \\ J_i^{(v_i)}=V}}
a^{\omega_2}_{V, (J_i)} U_{\omega_1}^k (\langle f_1, h_{J_1} \rangle_2, \langle f_2, h_{J_2}^0 \rangle_2) \otimes h_{J_3},
\end{equation}
where $\langle\cdot,\cdot\rangle_2$ stands for the inner product with respect to the variable $x_2$ ---recall that we always write $x=(x_1,x_2)$ with $x_1\in\R^n$ and $x_2\in\R^m$. Here it is useful to remember that the Haar functions live in $\R^m$, hence they only depend on $x_2$.

 Let us recall that by assumption $\vec{w}=(w_1,w_2) \in \mathbb{A}_{\vec{p}, n, m}$, see Definition \ref{def:LargerClass}. For $w=w_1 w_2$ we have by H\"older's inequality that
\begin{multline*}
\esssup_{x_1\in\re^n} \big[w(x_1, \cdot)^p ]_{A_{2p}(\R^m)}
\le
\esssup_{x_1\in\re^n} \big[w_1(x_1, \cdot)^{p_1}]_{A_{p_1}(\R^m)}^{\frac{p}{p_1}} \big[w_2(x_1, \cdot)^{p_2}]_{A_{p_2}(\R^m)}^{\frac{p}{p_2}}<\infty.
\end{multline*}
Hence $w(x_1, \cdot)^p \in A_\infty(\R^m)$ uniformly in $x_1\in \re^n$. With $x_1\in\R^n$ fixed, we can invoke Lemma \ref{lem:LowerSF} with $g_\omega := U_\omega(f_1,f_2)(x_1, \cdot)$ and the weight $w(x_1, \cdot)^p \in A_\infty(\R^m)$.  Then, we integrate over $x_1 \in \R^n$ and after some straightforward computations we arrive at
\begin{multline}\label{eq:ApplyLowerSF}
\iint \displaylimits _{\R^{n+m}}  |\E_\omega U_{\omega}(f_1,f_2)|^p w^p dx\\
\lesssim \iint \displaylimits _{\R^{n+m}} \Big( \E_{\omega} \Big [
\sum_{\substack{J_3, V \in \calD^m_{\omega_2} \\ J_3 ^{(v_3)}=V}}
\Big | \sum_{\substack{J_1, J_2 \in \calD^m_{\omega_2} \\ J_i^{(v_i)}=V}}
a^{\omega_2}_{V, (J_i)} U_{\omega_1}^k (\langle f_1, h_{J_1} \rangle_2, \langle f_2, h_{J_2}^0 \rangle_2) \otimes h_{J_3} \Big|^2 \Big]^{\frac{1}{2}} \Big)^p w^pdx.
\end{multline}
Let $s= (\frac{8m}{\alpha})'$ so that $\frac{m}{s'}=\frac{\alpha}{8}$ and clearly $1<s\le \frac87<2$. Let $J_3, V \in \calD^m_{\omega_2}$ be such that $J_3^{(v_3)}=V$.
From the size estimate of the coefficients $a^{\omega_2}_{V, (J_i)}$ we get that
\begin{multline}\label{eq:Intros}
\Big | \sum_{\substack{J_1, J_2 \in \calD^m_{\omega_2} \\ J_i^{(v_i)}=V}}
a^{\omega_2}_{V, (J_i)} U_{\omega_1}^k (\langle f_1, h_{J_1} \rangle_2, \langle f_2, h_{J_2}^0 \rangle_2) \otimes h_{J_3} \Big |
\\
\le 2^{\max_i v_i\frac{\alpha}{4}} \Big( \sum_{\substack{J_1, J_2 \in \calD^m_{\omega_2} \\ J_i^{(v_i)}=V}}
|U_{\omega_1}^k( f_{1,J_1,V}, f_{2,J_2,V}) |^s\Big)^{\frac{1}{s}} \otimes 1_{J_3},
\end{multline}
where we abbreviated $f_{1,J_1,V}:=\frac{|J_1|^{\frac12}}{|V|}\langle f_1, h_{J_1} \rangle_2$ and
$f_{2,J_2,V}:=\frac{|J_2|^{\frac12}}{|V|}\langle f_2, h_{J_2}^0 \rangle_2$.
Using this in \eqref{eq:ApplyLowerSF}
we have
\begin{multline}\label{weqfwferfgvea}
\Big(\iint \displaylimits _{\R^{n+m}}  |\E_\omega U_{\omega}(f_1,f_2)|^p w^p dx\Big)^{\frac{1}{p}}\\
\lesssim
2^{\max_iv_i \frac{\alpha}{4}} \Big(\iint \displaylimits _{\R^{n+m}} \Big( \E_{\omega} \Big [
\sum_{V \in \calD^m_{\omega_2}}
\Big ( \sum_{\substack{J_1, J_2 \in \calD^m_{\omega_2} \\ J_i^{(v_i)}=V}}
| U_{\omega_1}^k( f_{1,J_1,V}, f_{2,J_2,V}) |^s\Big)^{\frac2s} \otimes 1_{V} \Big]^{\frac12} \Big)^p w^p dx\Big)^{\frac 1p}.
\end{multline}

We next introduce some notation. Write $M^2 f(x) = M(f(x_1, \cdot))(x_2)$, where $M$ stands for the Hardy-Littlewood maximal function (over cubes) in $\R^m$ and
\begin{equation}\label{erfafer}
\Delta^2_{V,v_1} f_1(x) = \mathop{\sum_{J_1 \in \calD^m_{\omega_2}}}_{J_1^{(v_1)} = V} \Delta_{J_1}(f_1(x_1, \cdot))(x_2), \qquad V \in \calD^m_{\omega_2},
\end{equation}
and when $v_1=0$ we will simply write $\Delta_V^2$.
Recalling the definition of $f_{1,J_2,V}$ and \eqref{Haar-vs-Martingale} one can see that
\[
|f_{1,J_1,V}|
=
\frac{|J_1|^{\frac12}}{|V|}\big|\langle f_1, h_{J_1} \rangle_2\big|
=
\frac{|J_1|^{\frac12}}{|V|}\big|\langle \Delta^2_{V,v_1} f_1, h_{J_1} \rangle_2\big|
\le
\frac1{|V|}\int_{J_1} |\Delta^2_{V,v_1} f_1|dx_2,
\]
and, analogously,
\[
|f_{2,J_2,V}|
=
\frac{|J_2|^{\frac12}}{|V|}\big|\langle f_2, h_{J_2}^0 \rangle_2\big|
=
\frac{1}{|V|}\big|\langle f_2, 1_{J_2} \rangle_2\big|
\le
\frac1{|V|}\int_{J_2} |f_2|dx_2.
\]
Thus,
\begin{equation}\label{qrfafcwa}
\sum_{\substack{J_1 \in \calD^m_{\omega_2} \\ J_1^{(v_1)}=V}}   |f_{1,J_1,V}| \otimes 1_V
\le M^2 (\Delta^2_{V,v_1} f_1) \quad \text{ and } \quad
\sum_{\substack{J_2 \in \calD^m_{\omega_2} \\ J_2^{(v_2)}=V}} |f_{2,J_2,V}| \otimes1_V
\le M^2 f_2.
\end{equation}

Recall that $(w_1(\cdot,x_2),w_2(\cdot,x_2))\in A_{\vec{p}}(\R^n)$ for a.e. $x_2\in\re^m$ uniformly on $x_2$ (see Definition \ref{def:LargerClass}), and that $U_{\omega_1}^k$ satisfies the estimate \eqref{eq:ModelWeighted}. Therefore, we may
invoke Lemma \ref{lem:MZ} with the choice $q_1=2$, $q_2=\infty$, $q=2$ and with $s$ as above to conclude that
\begin{align*}
&\Big(\int_{\R^{n}} \Big( \E_{\omega} \Big [
\sum_{V \in \calD^m_{\omega_2}}
\Big ( \sum_{\substack{J_1, J_2 \in \calD^m_{\omega_2} \\ J_i^{(v_i)}=V}}
| U_{\omega_1}^k( f_{1,J_1,V}, f_{2,J_2,V}) |^s\Big)^{\frac2s} \otimes 1_{V}(x_2) \Big]^{\frac12} \Big)^p w(\cdot,x_2)^p dx_1\Big)^{\frac1p}
\\
&\lesssim
 (1+\max_i k_i)
\Big(\int _{\R^{n}} \Big (\E_{\omega_2} \sum_{V \in \calD^m_{\omega_2}}
\Big( \sum_{\substack{J_1 \in \calD^m_{\omega_2} \\ J_1^{(v_1)}=V}}
|f_{1,J_1,V}|^s \otimes 1_V(x_2) \Big)^\frac{2}{s} \Big)^{\frac {p_1}{2}} w_1(\cdot,x_2)^{p_1} dx_1\Big)^{\frac {1}{p_1}}
\\
&\qquad \times \Big( \int_{\R^{n}} \Big( \E_{\omega_2} \Big(\sup_{V \in \calD^m_{\omega_2}} \Big(
\sum_{\substack{J_2 \in \calD^m_{\omega_2} \\ J_2^{(v_2)}=V}}
|f_{2,J_2,V}|^s \otimes 1_V(x_2) \Big)^\frac{2}{s} \Big)^{\frac {p_2}{2}} w_2(\cdot,x_2)^{p_2} dx_1\Big)^{\frac {1}{p_2}}
\\
&\le
 (1+\max_i k_i)
\Big(\int _{\R^{n}} \Big (\E_{\omega_2} \sum_{V \in \calD^m_{\omega_2}}
\big( M^2 (\Delta^2_{V,v_1} f_1)  \big)^2 \Big)^{\frac {p_1}{2}} w_1(\cdot,x_2)^{p_1} dx_1\Big)^{\frac {1}{p_1}}
\\
&\qquad \times \Big( \int_{\R^{n}} (M^2 f_2)^{p_2} w_2(\cdot,x_2)^{p_2} dx_1\Big)^{\frac {1}{p_2}},
\end{align*}
where in the second estimate we have used that $s>1$ and \eqref{qrfafcwa}. This, Hölder's inequality, the fact that by assumption
$w_1(x_1, \cdot)^{p_1} \in A_{p_1}(\R^m)$ and $w_2(x_1, \cdot)^{p_2} \in A_{p_2}(\R^m)$ uniformly for a.e. $x_1 \in \R^n$ (see Definition \ref{def:LargerClass}), and Lemma \ref{lem:block-SF} readily give
\begin{align*}
&\Big(  \iint \displaylimits _{\R^{n+m}} \Big( \E_{\omega} \Big [
\sum_{V \in \calD^m_{\omega_2}}
\Big ( \sum_{\substack{J_1, J_2 \in \calD^m_{\omega_2} \\ J_i^{(v_i)}=V}}
| U_{\omega_1}^k( f_{1,J_1,V}, f_{2,J_2,V}) |^s\Big)^{\frac 2s} \otimes 1_{V} \Big]^{\frac 12} \Big)^p w^p dx\Big)^{\frac 1p} \\
&
\
\lesssim
(1+\max_i k_i)
\Big( \iint \displaylimits _{\R^{n+m}} \Big (\E_{\omega} \sum_{V \in \calD^m_{\omega_2}}
\big( M^2 (\Delta^2_{V,v_1} f_1)  \big)^2 \Big)^{\frac {p_1}{2}}  w_1^{p_1} dx \Big)^{\frac {1}{p_1}} \\
&\qquad \times \Big( \iint \displaylimits _{\R^{n+m}} (M^2 f_2)^{p_2}  w_2^{p_2} dx \Big)^{\frac {1}{p_2}}
\\
&
\
\lesssim
(1+\max_i k_i)
\|f_1 w_1\|_{L^{p_1}}\|f_2 w_2\|_{L^{p_2}}.
\end{align*}
Plugging this into \eqref{weqfwferfgvea} we arrive at \eqref{eq:TensModWeighted} in the case when $U_{\omega_1} \otimes U_{\omega_2}$ is of the form \eqref{eq:SpecialShift}.

\medskip

Let us now consider the cases in which $ U_{\omega_2}^v$ is a bilinear dyadic shift of some other form.
The case in which the non-cancellative Haar function is located in the ``first'' position (i.e., our operator is written in terms of
$h_{J_1}$, $h^0_{J_2}$ and $h_{J_3}$) is clearly symmetric to the case we treated already (one just needs to switch the roles of $f_1$ and $f_2$). Thus, we only need to consider the case when
\begin{equation}\label{eq:SpecialShift2}
U_{\omega}(f_1,f_2)
=\sum_{V \in \calD^m_{\omega_2}} \sum_{\substack{J_1, J_2,J_3 \in \calD^m_{\omega_2} \\ J_i^{(v_i)}=V}}
a^{\omega_2}_{V, (J_i)} U_{\omega_1}^k (\langle f_1, h_{J_1} \rangle_2, \langle f_2, h_{J_2} \rangle_2) \otimes h^0_{J_3}.
\end{equation}
This time, rather than using the lower bound for the square function estimate, we just proceed as in \eqref{eq:Intros}, putting absolute values inside and using H\"older's inequality with  $s$ as before.
This gives that
\begin{equation*}
|U_{\omega}(f_1,f_2)|
\le
 2^{\max v_i \alpha /4}
\sum_{V \in \calD^m_{\omega_2}}
\Big( \sum_{\substack{J_1, J_2 \in \calD^m_{\omega_2} \\ J_i^{(v_i)}=V}}
|U_{\omega_1}^k( f_{1,J_1,V}, f_{2,J_2,V}) |^s\Big)^{\frac{1}{s}} \otimes 1_{V},
\end{equation*}
where $f_{1,J_1,V}:=\frac{|J_1|^{\frac12}}{|V|}\langle f_1, h_{J_1} \rangle_2$ and
$f_{2,J_2,V}:=\frac{|J_2|^{\frac12}}{|V|}\langle f_2, h_{J_2} \rangle_2$. We can proceed much as before, with the difference that Lemma \ref{lem:MZ} is used with the exponents $q_1=q_2=2$, $q=1$ and also that since $h_{J_2}$ is cancellative we have the analog of the first estimate in \eqref{qrfafcwa} for $f_{2,J_2,V}$. This gives that
\begin{align}\label{eq:UseMZ2}
&\|\E_\omega   U_{\omega}(f_1,f_2)w\|_{L^p}
\\
&\qquad\lesssim  2^{\max v_i \alpha /4}  (1+\max_i k_i)
\Big\|\Big (\E_{\omega} \sum_{V \in \calD^m_{\omega_2}}
\big( M^2 (\Delta^2_{V,v_1} f_1)  \big)^2 \Big)^{\frac {1}{2}}  w_1\Big\|_{p_1} \\
&\quad\qquad\qquad \times
\Big\|\Big (\E_{\omega} \sum_{V \in \calD^m_{\omega_2}}
\big( M^2 (\Delta^2_{V,v_2} f_2)  \big)^2 \Big)^{\frac {1}{2}}  w_1\Big\|_{p_2}
\\
&\qquad\lesssim  2^{\max v_i \alpha /4}  (1+\max_i k_i)
\|f_1 w_1\|_{L^{p_1}}\|f_2 w_2\|_{L^{p_2}},
\end{align}
where in the last estimate we have used that $w_1(x_1, \cdot)^{p_1} \in A_{p_1}(\R^m)$ and $w_2(x_1, \cdot)^{p_2} \in A_{p_2}(\R^m)$ uniformly for a.e. $x_1 \in \R^n$ (see Definition \ref{def:LargerClass}), and Lemma \ref{lem:block-SF}. This ends our treatment of the case when $U_{\omega_2}$ is a dyadic bilinear shift.

\medskip

\noindent\textbf{Case 2:} Proof of \eqref{eq:TensModWeighted} when $U^v_{\omega_2,u_2}$ is a dyadic bilinear paraproduct.

As before to simplify the notation we write  $U^k_{\omega_1}=U^k_{\omega_1,u_1}$ and $U^v_{\omega_2}=U^v_{\omega_2,u_2}$. Here we assume that $U_{\omega_2}^v$ is of the form $\mathcal{U}_{\calD_{\omega_2}^m}$, see \eqref{dyadic-para}, and hence  $v=(0,0,0)$. On the other hand,  $U_{\omega_1}^k$ is any model operator. The argument for this case is easier than the previous one: before we used  the Marcinkiewicz-Zygmund type estimate in
Lemma \ref{lem:MZ} to deal with the complexity of $U_{\omega_2}^v$ and in this scenario the complexity is zero.

Recall that dyadic bilinear paraproducts have three different expressions depending on where the cancellative function is located. We first consider the case where

\[
U^v_{\omega_2}(g_1,g_2)
=
\sum_{V \in \calD^m_{\omega_2}} a^{\omega_2}_{V}\langle g_1\rangle_V \langle g_2\rangle_V h_{V},
\]
and the other cases will be treated below.

We write as before  $U_\omega:= U^k_{\omega_1} \otimes U^v_{\omega_2}$ so that
\begin{equation}\label{eq:SpecialPara}
U_\omega(f_1,f_2)
= \sum_{V \in \calD^m_{\omega_2}} a_V^{\omega_2}
U_{\omega_1}^k ( \langle f_1 \rangle_{V,2}, \langle f_2 \rangle_{V,2}) \otimes h_V.
\end{equation}
As shown in the previous case $w_1(x_1, \cdot)^p w_2(x_1, \cdot)^p \in A_\infty(\R^m)$ uniformly in $x_1\in \re^n$.  With $x_1\in\R^n$ fixed, we can use Lemma \ref{lem:LowerSF} with $g_\omega := U_\omega(f_1,f_2)(x_1, \cdot)$ and the weight $w_1(x_1, \cdot)^p w_2(x_1, \cdot)^p \in A_\infty(\R^m)$.
After that we integrate on $x_1$ and arrive at
\[
\| \E_\omega U_\omega(f_1,f_2) w\|_{L^p}
\lesssim \Big \| \E_{\omega} \Big( \sum_{V \in \calD^m_{\omega_2}}  | a_V^{\omega_2}
U_{\omega_1} ( \langle f_1 \rangle_{V,2}, \langle f_2 \rangle_{V,2}) \otimes h_V |^2 \Big)^{\frac 12} w\Big \|_{L^p}.
\]
Recall that $(w_1(\cdot,x_2),w_2(\cdot,x_2))\in A_{\vec{p}}(\R^n)$ for a.e. $x_2\in\re^m$ uniformly on $x_2$ (see Definition \ref{def:LargerClass}) and that  $U_{\omega_1}$ satisfies \eqref{eq:ModelWeighted}. Thus, we may with fixed $x_2\in\R^m$
invoke Lemma \ref{lem:MZ} with the choice $q_1=2$, $q_2=\infty$ and $q=2$  in the very special case where the inner $\ell^s$-sums have only one non-zero term. This gives that
\begin{align*}
&\Big(\int_{\R^{n}} \Big( \E_{\omega} \Big [
\sum_{V \in \calD^m_{\omega_2}}
|a_V^{\omega_2}  U_{\omega_1}^k ( \langle f_1 \rangle_{V,2}, \langle f_2 \rangle_{V,2})\otimes h_{V}(x_2)|^2 \Big]^{\frac12} \Big)^p w(\cdot,x_2)^p dx_1\Big)^{\frac1p}
\\
&\ \lesssim
 (1+\max_i k_i)
\Big(\int _{\R^{n}} \Big (\E_{\omega_2} \sum_{V \in \calD^m_{\omega_2}}
|a_V^{\omega_2} \langle f_1 \rangle_{V,2}\otimes h_{V}(x_2)|^2\Big)^{\frac {p_1}{2}} w_1(\cdot,x_2)^{p_1} dx_1\Big)^{\frac {1}{p_1}}
\\
&\qquad \times \Big( \int_{\R^{n}} \Big( \E_{\omega_2} \sup_{V \in \calD^m_{\omega_2}} |\langle f_2 \rangle_{V,2}\otimes 1_{V}(x_2)|^2\Big)^{\frac {p_2}{2}} w_2(\cdot,x_2)^{p_2} dx_1\Big)^{\frac {1}{p_2}}
\\
&\ \lesssim
(1+\max_i k_i)
\Big(\int _{\R^{n}} \Big (\E_{\omega_2} \sum_{V \in \calD^m_{\omega_2}}
|a_V^{\omega_2} \langle f_1 \rangle_{V,2}\otimes h_{V}(x_2)|^2\Big)^{\frac {p_1}{2}} w_1(\cdot,x_2)^{p_1} dx_1\Big)^{\frac {1}{p_1}}
\\
&\qquad \times \Big( \int_{\R^{n}} (M^2 f_2)^{p_2} w_2(\cdot,x_2)^{p_2} dx_1\Big)^{\frac {1}{p_2}},
\end{align*}
where we have used the pointwise inequality
$\sup_{V \in \calD^m_{\omega_2}}
|\langle f_2 \rangle_{V,2}| \otimes 1_V
\le M^2f_2$. Using the obtained estimates, Hölder's inequality, the fact that by assumption
$w_1(x_1, \cdot)^{p_1} \in A_{p_1}(\R^m)$ and $w_2(x_1, \cdot)^{p_2} \in A_{p_2}(\R^m)$ uniformly for a.e. $x_1 \in \R^n$ (see Definition \ref{def:LargerClass}), and Lemma \ref{lemma:paraprod} we conclude that
\begin{multline*}
\| \E_\omega U_\omega(f_1,f_2) w\|_{L^p}
\lesssim
 (1+\max_i k_i) \Big \| \Big( \E_\omega \sum_{V \in \calD^m_{\omega_2}}
|a_V^{\omega_2} \langle f_1 \rangle_{V,2}|^2\otimes \frac{1_V}{|V|}  \Big)^{\frac 12}w_1 \Big \|_{L^{p_1}}
\\
\times\Big( \iint \displaylimits _{\R^{n+m}} (M^2 f_2)^{p_2}  w_2^{p_2} dx \Big)^{\frac {1}{p_2}}
\lesssim
(1+\max_i k_i)
\|f_1 w_1\|_{L^{p_1}}  \|f_2 w_2\|_{L^{p_2}}.
\end{multline*}
This gives the desired estimate in the present case. Let us finally consider the remaining two possibilities for  $U_{\omega_2}^v$. By symmetry, we may suppose that
\[
U_\omega(f_1,f_2)
= \sum_{V \in \calD^m_{\omega_2}} a_V^{\omega_2}
U_{\omega_1}^k ( \langle f_1\rangle_{2}, \langle f_2, h_V  \rangle_{V,2}) \otimes \frac{1_V}{|V|}.
\]
and the other cases is treated similarly switching the roles of $f_1$ and $f_2$. This time we do not use the lower square function estimate and directly apply Lemma \ref{lem:MZ} with exponents $q_1=q_2=2$, $q=1$ and again the inner $\ell^s$-sums have only one non-zero term. Hence, much as before we obtain
\begin{multline*}
\| \E_\omega U_\omega(f_1,f_2) w\|_{L^p}
\lesssim
 (1+\max_i k_i)
\Big \| \Big( \E_\omega \sum_{V \in \calD^m_{\omega_2}}
|a_V^{\omega_2} \langle f_1 \rangle_{V,2}|^2\otimes \frac{1_V}{|V|}  \Big)^{\frac 12}w_1 \Big \|_{L^{p_1}}
\\
\times \Big \| \Big( \E_\omega \sum_{V \in \calD^m_{\omega_2}}
|\langle f_2,h_V \rangle_{V,2}|^2\otimes \frac{1_V}{|V|}  \Big)^{\frac 12}w_2 \Big \|_{L^{p_2}}
\\
\lesssim
 (1+\max_i k_i) \|f_1 w_1\|_{L^{p_1}}  \|f_2 w_2\|_{L^{p_2}},
\end{multline*}
where the term corresponding to $f_1$ has been treated as before. For $f_2$, using the notation introduced in \eqref{erfafer}), one first sees that
\eqref{Haar-vs-Martingale} gives
\begin{multline*}
|\langle f_2,h_V \rangle_{V,2}|\otimes \frac{1_V}{|V|^{\frac12}}
=
|\langle \Delta_V^2 f_2,h_V \rangle_{V,2}|\otimes \frac{1_V}{|V|^{\frac12}}
\\
\le
\frac1{|V|}\int_{V} |\Delta^2_{V} f_2|dx_2\otimes 1_V
\le
M^2 (\Delta^2_{V} f_2)
\end{multline*}
and therefore Lemma \ref{lem:block-SF} with $k=0$, along with the fact that $w_2(x_1, \cdot)^{p_2} \in A_{p_2}(\R^m)$ uniformly for a.e. $x_1 \in \R^n$ (see Definition \ref{def:LargerClass}), yields
\begin{multline*}
\Big \| \Big( \E_\omega \sum_{V \in \calD^m_{\omega_2}}
|\langle f_2,h_V \rangle_{V,2}|^2\otimes \frac{1_V}{|V|}  \Big)^{\frac 12}w_2 \Big \|_{L^{p_2}}
\\
\le
\Big \| \Big( \E_\omega \sum_{V \in \calD^m_{\omega_2}} (M^2 (\Delta^2_{V} f_2))^2 \Big)^{\frac 12}w_2 \Big \|_{L^{p_2}}
\lesssim
\|f_2w_2\|_{L^{p_2}}.
\end{multline*}
This concludes the proof of \eqref{eq:TensModWeighted}, and hence that of Theorem \ref{thm:WeightedTensor}. \qed

\subsection{Proof of Corollary \ref{cor:TensorEndpoint}}

We proceed by extrapolation. Fix $1<q_1,q_2<\infty$ and let $\frac1q=\frac1{q_1}+\frac1{q_2}> 0$. Given $f_1, f_2 \in L^{\infty}_c(\R^{n+m})$ define for every $x_1\in\R^n$ and for $i=1,2$
\begin{equation}\label{afawfwf}
F(x_1):= \| T_n \otimes T_m (f_1,f_2)(x_1, \cdot) \|_{L^q(\R^m)}
\quad \text{and} \quad  F_i(x_1):= \| f_i(x_1, \cdot) \|_{L^{q_i}(\R^m)}
\end{equation}
Suppose $\vec{w}=( w_1, w_2) \in A_{\vec{q}}(\R^n)$ with $\vec{q}=(q_1,q_2)$ and as usual we define $w=w_1w_2$. It is easy to see that $(w_1\otimes 1, w_2\otimes 1)\in \mathbb A_{\vec{q}, n, m}$.
From Theorem \ref{thm:WeightedTensor} ---here it is crucial that $1<q_1,q_2<\infty$--- we see that
\begin{align}\label{22qffd}
\| F w \|_{L^q(\R^n)}
&= \| T_n \otimes T_m (f_1,f_2)  (w\otimes 1) \|_{L^q(\R^{n+m})} \\ \nonumber
&\lesssim
\| f_1 (w_1\otimes 1) \|_{L^{q_1}(\R^{n+m})}\| f_2 (w_2\otimes 1) \|_{L^{q_2}(\R^{n+m})} \\ \nonumber
&=\| F_1 w_1 \|_{L^{q_1}(\R^n)} \| F_2 w_2 \|_{L^{q_2}(\R^n)}.
\end{align}
We can extrapolate from this estimate with the family $\mathcal{F}$ consisting on the triples $(F,F_1, F_2)$ and Theorem \ref{theor:extrapol-general} (with $\vec{r}=(1,1, 1)$) gives with $\vec{p}=(p_1,p_2)$ and $1<p_1,p_2\le\infty$ with $\frac1p=\frac1{p_1}+\frac1{p_2}>0$ that
$$
\| F w \|_{L^p(\R^n)}
\lesssim \| F_1 w_1 \|_{L^{p_1}(\R^n)} \| F_2 w_2 \|_{L^{p_2}(\R^n)}
$$
for all $(w_1,w_2) \in A_{\vec{p}}(\R^n)$.
This in the special case $w_1 \equiv 1$ and $w_2 \equiv 1$ gives us the estimate
\begin{multline}\label{eq:eq1afvrv}
\| T_n \otimes T_m (f_1,f_2) \|_{L^{p}(\R^n; L^{q}(\R^m))}
=
\| F\|_{L^p(\R^n)}
\\
\lesssim
\| F_1 \|_{L^{p_1}(\R^n)} \| F_2 \|_{L^{p_2}(\R^n)}
\lesssim \|f_1 \|_{L^{p_1}(\R^n; L^{q_1}(\R^m))}
\| f_2 \|_{L^{p_2}(\R^n; L^{q_2}(\R^m))}
\end{multline}
for all  $f_1, f_2 \in L^{\infty}_c(\R^n; L^{q_1}(\R^m))$. This ends the proof of Corollary \ref{cor:TensorEndpoint}.
\qed

\subsection{Proof of Theorem \ref{theor:TensorEndpoint}}

We first observe that a trivial density argument allows us to see that Corollary \ref{cor:TensorEndpoint} readily gives the desired estimate for all $f_1\in {L^{p_1}(\R^n; L^{q_1}(\R^m))}$ and $f_2\in {L^{p_2}(\R^n; L^{q_2}(\R^m))}$ provided $1<p_1,p_2,q_1,q_2<\infty$. This means that we only need to consider the cases $q_1=\infty$ or $q_2=\infty$ and/or  $p_1=\infty$ or $p_2=\infty$. In these scenarios we need to justify why the operators are well-defined and also get the desired estimates. To accomplish all these we split the argument in two main steps which are almost independent, and each of them is interesting in its own right. In \textbf{Case 1} we use sparse domination techniques and the main goal is to treat the cases $q_1=\infty$ or $q_2=\infty$ on which our main extrapolation result is not useful. However, we prove more, mostly because in \textbf{Case 2} we will need to know that the operators are well-defined for some class of functions. As a result, \textbf{Case 1} deals with the exponents  $1<p_1,p_2\le\infty$ with $\frac1p=\frac1{p_1}+\frac1{p_2}>0$ and $1<q_1,q_2\le \infty$ with $0<\frac1q=\frac1{q_1}+\frac1{q_2}\le 1$. We note that the restriction $q\ge 1$ is natural since we use sparse domination and duality, hence we need to be in the Banach range at least for $q$. In this direction, if $q_1=\infty$ (resp. $q_2=\infty$) then $q=q_2>1$ (resp. $q=q_1>1$), thus in these two cases we do not have a real restriction. In \textbf{Case 2}, where $1<q_1,q_2<\infty$, we are going to use our extrapolation result again, Theorem \ref{theor:extrapol-general}, and in doing so we need to take a detour to show that the operators are well-defined and this is where \textbf{Case 1} is invoked in a particular case where the $q_i$'s are finite. It is important to note that at the end-point cases (i.e., when some of the $p_i$'s or $q_i$'s are infinity) we are able to obtain the desired estimates for functions in the corresponding spaces and not just in $L^\infty_c(\R^{n+m})$. This requires to justify why the operators are well-defined and we do this by duality using the two adjoints of our operators. This makes some of our argument tedious needing to consider several cases quite carefully. If one were interested in obtaining the desired estimates just for compactly supported functions one could skip some of those cases.

\medskip

\noindent\textbf{Case 1:}  $1<p_1,p_2\le\infty$ with $\frac1p=\frac1{p_1}+\frac1{p_2}>0$ and $1<q_1,q_2\le \infty$ with $0<\frac1q=\frac1{q_1}+\frac1{q_2}\le 1$.

We first prove  certain vector-valued sparse domination result using dyadic model operators.
The starting point is that as observed above we already know that for $f_1, f_2, f_3 \in L^{\infty}_c(\R^{n+m})$ we have
\[
\langle T_n \otimes T_m(f_1, f_2), f_3 \rangle = C_{T_n}C_{T_m}
\E_\omega \sum_{k \in \N_0^3} \sum_{v \in \N_0^3}   \sum_{u_1=1}^{C_n}\sum_{u_2=1}^{C_m}
c_{k,v} \langle U^k_{\omega_1,u_1} \otimes U^v_{\omega_2,u_2} (f_1,f_2), f_3 \rangle,
\]
where $c_{k,v}:= 2^{-\max_i k_i \frac{\alpha_1}2 -\max_i v_i \frac{\alpha_2}2}$ and $\alpha_1$ and $\alpha_2$ are respectively
the parameters in the H\"older continuity assumptions of the kernels of $T_n$ and $T_m$.
Denote $U_\omega:= U_{\omega_1, u_1}^{k} \otimes U_{\omega_2, u_2}^{v}$.

Then, with $1<q_1, q_2\le \infty$ so that $0<1/q := 1/q_1 + 1/q_2\le 1$ and for any $f_1, f_2, f_3 \in L^{\infty}_c(\R^{n+m})$ we write
$$
\langle U_{\omega}(f_1,f_2), f_3\rangle
=\sum_{K \in \calD^n_{\omega_1, u_1}} \sum_{\substack{I_1, I_2, I_3 \in \calD^n_{\omega_1, u_1} \\ I_i^{(k_i)}=K}}
a^{\omega_1, u_1}_{K, (I_i)} \langle U_{\omega_2,u_2}^v (\langle f_1, \tilde h_{I_1} \rangle_1, \langle f_2, \tilde h_{I_2} \rangle_1), \langle f_3, \tilde h_{I_3}\rangle_1 \rangle,
$$
where $\tilde h_{I} \in \{h_I, h_I^0\}$.
It is not hard to adapt the bilinear sparse domination argument of model operators \cite[Section 5]{LMOV}
and deduce that for all $f_1, f_2, f_3 \in L^{\infty}_c(\R^{n+m})$ we have
\begin{multline*}
|\langle U_{\omega}(f_1,f_2), f_3\rangle|
\lesssim (1+\max_i k_i) \Lambda_\mathcal{S}(\|f_1\|_{L^{q_1}(\R^m)}, \|f_2\|_{L^{q_2}(\R^m)}, \|f_3\|_{L^{q'}(\R^m)})
\\
:=
(1+\max_i k_i)
\sum_{Q\in\mathcal{S}}|Q| \bla \|f_1\|_{L^{q_1}(\R^m)} \bra_Q \bla \|f_2\|_{L^{q_2}(\R^m)} \bra_Q \bla \|f_3\|_{L^{q'}(\R^m)} \bra_Q
\end{multline*}
for some dyadic grid $\calD^n$ and sparse collection of cubes $\calS \subset \calD^n$ depending on
$f_1$, $f_2$, $f_3$, $q_1$, $q_2$. Here it is important to notice
that $\calS$ and $\calD^n$ do not depend on $\omega$ (see Lacey-Mena \cite{LM} or \cite{LMOV}).
Thus,
\begin{equation}\label{q334r3f}
|\langle T_n \otimes T_m(f_1,f_2), f_3\rangle|
\le
 C \Lambda_\mathcal{S}(\|f_1\|_{L^{q_1}(\R^m)}, \|f_2\|_{L^{q_2}(\R^m)}, \|f_3\|_{L^{q'}(\R^m)}).
\end{equation}

\medskip

\noindent\textbf{Case 1a:} $q_1,q_2<\infty$ and $p_1,p_2<\infty$.

This case follows at once from Corollary \ref{cor:TensorEndpoint} and a standard density argument.

\medskip \noindent\textbf{Case 1b:} $q_1,q_2<\infty$ and $p_1=\infty$ or $p_2=\infty$.

By symmetry we may assume that $p_1=\infty$ and hence $p_2=p$ with $1<p<\infty$.  We proceed by duality and observe first that $T_n^{1*}$ (resp. $T_m^{1*}$), the adjoint with respect the first entries of $T_n$ (resp. $T_m$) is a bilinear Calder\'on-Zygmund operator in $\R^n$ (resp. $\R^m$). Also, one can see that $(T_n \otimes T_m)^{1*}=T_n^{1*}\otimes T_m^{1*}$ and therefore
the already established estimates give
\[
\|(T_n \otimes T_m)^{1*}(f_3,f_2)\|_{L^{1}(\R^n;L^{q_1'}(\R^m))}  \lesssim \|f_3\|_{L^{p'}(\R^n; L^{q'}(\R^m))} \|f_2\|_{L^{p}(\R^n;L^{q_2}(\R^m))}
\]
for all $f_3\in L^{p'}(\R^n; L^{q'}(\R^m))$ and $f_2\in L^{p}(\R^n;L^{q_2}(\R^m))$ since in the current case  we have $1<p',p<\infty$, $1<q',q_2<\infty$, and $\frac1{q_1'}=\frac1{q'}+\frac1{q_2}<1$.

In turn for every $f_1\in {L^{\infty}(\R^n; L^{q_1}(\R^m))}$ and $f_2\in {L^{p}(\R^n; L^{q_2}(\R^m))}$ we can define $T_n \otimes T_m(f_1, f_2) \in L^p(\R^n; L^q(\R^m))$ by setting for all $f_3 \in L^{p'}(\R^n; L^{q'}(\R^m))$
$$
\langle T_n \otimes T_m(f_1,f_2),f_3 \rangle
:= \langle (T_n \otimes T_m)^{1*}(f_3,f_2),f_1 \rangle.
$$
All these eventually show that \eqref{qfwafaerfer:theor} holds and \textbf{Case 1b} is complete.

\medskip

\noindent\textbf{Case 1c:} $q_1=\infty$ or $q_2=\infty$ and $p_1,p_2<\infty$.

By symmetry we may assume that $q_1=\infty$, hence $1< q=q_2<\infty$. Our first goal is to see that the above sparse estimate \eqref{q334r3f} holds for any
 $f_1 \in L^\infty(\R^{n+m})$ and  $f_2, f_3 \in L^\infty_c(\R^{n+m})$. As observed before $ (T_n \otimes T_m)^{1*} (f_3,f_2)= T_n^{1*} \otimes T_m^{1*} (f_3,f_2)$ and
by Case 1a with exponents $\tilde{p}_1=\tilde{p}_2=\tilde{q}_1=\tilde{q}_2=2$ we have
\begin{multline*}
\|(T_n \otimes T_m)^{1*}(f_3,f_2)\|_{L^{1}(\R^{n+m})}
=
\|(T_n \otimes T_m)^{1*}(f_3,f_2)\|_{L^{1}(\R^n;L^{1}(\R^m))}
\\
\lesssim \|f_3\|_{L^{2}(\R^n; L^{2}(\R^m))} \|f_2\|_{L^{2}(\R^n;L^{2}(\R^m))}<\infty.
\end{multline*}
Let $R_1 \subset R_2 \subset \cdots$ be an increasing sequence of rectangles such that $\bigcup_k R_k= \R^{n+m}$ and note that by duality
\begin{multline*}
\langle T_n \otimes T_m (f_1,f_2), f_1 \rangle
=\langle (T_n \otimes T_m)^{1*} (f_3,f_2), f_1 \rangle
\\
=\lim_{k \to \infty} \langle (T_n \otimes T_m)^{1*} (f_3,f_2),1_{R_k} f_1 \rangle
=\lim_{k \to \infty} \langle T_n \otimes T_m (1_{R_k}f_1 , f_2), f_3 \rangle.
\end{multline*}
Hence since all the quantities involved are finite we can choose $k$ (depending on $f_1,f_2,f_3$) so that
\[
|\langle T_n \otimes T_m (f_1, f_2), f_3 \rangle|
\le 2 |\langle T_n \otimes T_m (1_{R_k}f_1, f_2), f_3 \rangle|.
\]
At this point we can invoke \eqref{q334r3f} ---which is valid with $q_1=\infty$, $q_2=q$ as observed above--- for $1_{R_k}f_1, f_2, f_3\in  L^\infty_c(\R^{n+m})$ and find a sparse family $\mathcal{S}$ and a dyadic grid $\mathcal{D}^n$ with $\mathcal{S}\subset\mathcal{D}^n$ (both $\mathcal{S}$ and $\mathcal{D}^n$ depending on $f_1,f_2,f_3$, and on $k$ which ultimately depends on these three functions) such that
\begin{multline*}
|\langle T_n \otimes T_m (f_1, f_2), f_3 \rangle|
\lesssim
\Lambda_\mathcal{S}(\|1_{R_k} f_1\|_{L^{\infty}(\R^m)}, \|f_2\|_{L^{q}(\R^m)}, \|f_3\|_{L^{q'}(\R^m)})
\\
\le
\Lambda_\mathcal{S}(\|f_1\|_{L^{\infty}(\R^m)}, \|f_2\|_{L^{q}(\R^m)}, \|f_3\|_{L^{q'}(\R^m)}).
\end{multline*}
All in one, we have been able to show that for any $f_1 \in L^\infty(\R^{n+m})$, $f_2, f_3 \in L^\infty_c(\R^{n+m})$ there is a sparse domination formula as in \eqref{q334r3f} with a possible larger constant.

To proceed, let $w\in A_\infty(\R^n)$ and
note that for any $f_1 \in L^\infty(\R^{n+m})$ and $f_2\in L^{\infty}_c(\R^{n+m})$ we have
\begin{multline*}
\|T_n \otimes T_m(f_1,f_2)\|_{L^q(w(x_1)dx_1,\R^n; L^q(\R^m))}
=
\|T_n \otimes T_m(f_1,f_2)\|_{L^q(w\otimes 1,\R^{n+m})}
\\
=
\sup
|\langle T_n \otimes T_m(f_1,f_2), f_3\,(w\otimes 1)^{\frac1q}\rangle|,
\end{multline*}
where the sup runs over all $f_3\in L^\infty_c(\R^{n+m})$ with $\|f_3\|_{L^{q'}}=1$. Fix  such a function $f_3$ and write $\tilde{f}_1=\|f_1\|_{L^{\infty}(\R^m)}$, $\tilde{f}_2=\|f_2\|_{L^{q}(\R^m)}$, and $\tilde{f}_3=\|f_3\|_{L^{q_3}(\R^m)}$. By the previous argument we know that there exists a sparse family $\mathcal{S}$ and a dyadic grid $\mathcal{D}^n$ with $\mathcal{S}\subset\mathcal{D}^n$ (both $\mathcal{S}$ and $\mathcal{D}^n$ depending on $f_1,f_2,f_3$) such that
\begin{align*}
&|\langle T_n \otimes T_m(f_1,f_2), f_3\,(w\otimes 1)^{\frac1q}\rangle|
\lesssim
\Lambda_\mathcal{S}(\tilde{f}_1, \tilde{f}_2, \tilde{f}_3 \,w^{\frac1q})
\\
&\qquad=
\sum_{Q\in\mathcal{S}} w(Q)  \bla \tilde{f}_1 \bra_Q \bla \tilde{f}_2\bra_Q \frac1{w(Q)}\int_Q (\tilde{f}_3 w^{-\frac1{q'}}) wdx_1
\\
&\qquad\lesssim
\sum_{Q\in\mathcal{S}} w(E_Q)  \bla \tilde{f}_1 \bra_Q \bla \tilde{f}_2\bra_Q \frac1{w(Q)}\int_Q (\tilde{f}_3 w^{-\frac1{q'}}) wdx_1
\\
&\qquad\lesssim
\int_{\R^n} M(\tilde{f}_1,\tilde{f}_2) M^{\mathcal{D}^n}_{w}  (\tilde{f}_3 w^{-\frac1{q'}})  wdx_1
\\
&\qquad\lesssim
\|M(\tilde{f}_1,\tilde{f}_2)\|_{L^q(w,\R^n)}\|M^{\mathcal{D}^n}_{w}  (\tilde{f}_3 w^{-\frac1{q'}})\|_{L^{q'}(w,\R^n)}
\\
&\qquad
\lesssim
\|M(\tilde{f}_1,\tilde{f}_2)\|_{L^q(w,\R^n)}.
\end{align*}
In the previous computations we have used that $\{E_Q\}_{Q\in\mathcal{S}}$ is the pairwise disjoint family associated with the sparse family $\mathcal{S}$ for which we have $|E_Q|\approx |Q|$. Also, since $w\in A_\infty$ it follows that $w(Q)\approx w(E_Q)$ since $|E_Q|\approx |Q|$. Finally, $M^{\mathcal{D}^n}_{w}$ is the $\mathcal{D}^n$-dyadic Hardy-Littlewood maximal function with underlying measure $wdx_1$ which is bounded in $L^{q'}(w,\R^n)$ for every $1\le q<\infty$. Gathering all the obtained estimates we have concluded that for all $w\in A_\infty(\R^n)$
\begin{multline*}
\big\| \|T_n \otimes T_m(f_1,f_2)\|_{L^q(\R^m)}\big\|_{L^q(w,\R^n)}
=
\|T_n \otimes T_m(f_1,f_2)\|_{L^q(w(x_1)dx_1,\R^n; L^q(\R^m))}
\\
\lesssim
\|M(\tilde{f}_1,\tilde{f}_2)\|_{L^q(w,\R^n)}.
\end{multline*}
Using then $A_\infty$-extrapolation, see  \cite[Theorem 2.1]{CUMP}, we obtain for every $0<r<\infty$
\begin{multline*}
\|T_n \otimes T_m(f_1,f_2)\|_{L^r(\R^n; L^q(\R^m))}
=
\big\| \|T_n \otimes T_m(f_1,f_2)\|_{L^q(\R^m)}\big\|_{L^r(\R^n)}
\\
\lesssim
\|M(\tilde{f}_1,\tilde{f}_2)\|_{L^r(\R^n)}.
\end{multline*}
At this point, given $1<p_1,p_2<\infty$ with $\frac1p=\frac1{p_1}+\frac1{p_2}>0$ we can apply the previous estimate with $r=p$ and the well-known estimates for the bilinear maximal operator (see \cite{LOPTT}) to conclude that
\begin{multline}\label{r43fvr}
\|T_n \otimes T_m(f_1,f_2)\|_{L^p(\R^n; L^q(\R^m))}
\lesssim
\|M(\tilde{f}_1,\tilde{f}_2)\|_{L^p(\R^n)}
\lesssim
\|\tilde{f}_1\|_{L^{p_1}(\R^n)} \|\tilde{f}_2\|_{L^{p_2}(\R^n)}
\\
=
\|f_1\|_{L^{p_1}(\R^n; L^{\infty}(\R^m))} \|f_2\|_{L^{p_2}(\R^{n}; L^{q}(\R^m))}.
\end{multline}
for all $f_1 \in L^\infty(\R^{n+m})$ and $f_2\in L^{\infty}_c(\R^{n+m})$.

Next, given $f_1\in {L^{p_1}(\R^n; L^{\infty}(\R^m))}$  and $f_2\in L^{\infty}_c(\R^{n+m})$ consider $f_1^N(x_1,x_2)= f_1(x_1,x_2)1_{\{x_1: \|f(x_1,\cdot)\|_{L^\infty(\R^m)}<N\}}$ for every $(x_1,x_2)\in\R^{n+m}$ and $N\ge 1$. It is straightforward to see that $\{f_1^N\}_{N\ge 1}$ is Cauchy sequence in
${L^{p_1}(\R^n; L^{\infty}(\R^m))}$ and hence this allows us to define
\[
T_n \otimes T_m(f_1,f_2):=\lim_{N\to\infty} T_n \otimes T_m(f_1^N,f_2)
\]
where the convergence is in ${L^p(\R^n; L^q(\R^m))}$. Since $f_1^N\in L^\infty(\R^{n+m})$ and $f_2\in L^{\infty}_c(\R^{n+m})$  we can invoke \eqref{r43fvr} to see that \eqref{qfwafaerfer:theor} holds for all $f_1\in {L^{p_1}(\R^n; L^{\infty}(\R^m))}$  and $f_2\in L^{\infty}_c(\R^{n+m})$. Next, since $p_2,q<\infty$ we can readily extend this estimate by density to all
$f_1\in {L^{p_1}(\R^n; L^{\infty}(\R^m))}$ and $f_2\in {L^{p_2}(\R^n; L^{q}(\R^m))}$  and this completes the proof of \eqref{qfwafaerfer:theor} in the present case.

\medskip

\noindent\textbf{Case 1d:} $q_1=\infty$ or $q_2=\infty$, and $p_1=\infty$ or $p_2=\infty$.

Again by symmetry we may assume that $q_1=\infty$, hence $1< q=q_2<\infty$. Consider first the case $p_1=\infty$, thus $1<p_2=p<\infty$. As before,
from \textbf{Case 1a} we readily get
\begin{align*}
\|(T_n \otimes T_m)^{1*}(f_3,f_2)\|_{L^{1}(\R^n;L^{1}(\R^m))}
\lesssim
\|f_3\|_{L^{p'}(\R^n; L^{q'}(\R^m))} \|f_2\|_{L^{p}(\R^n;L^{q}(\R^m))},
\end{align*}
for every $f_2\in L^{p}(\R^n;L^{q}(\R^m))$ and $f_3 \in L^{p'}(\R^n; L^{q'}(\R^m))$.
Then for every $f_1\in {L^{\infty}(\R^n; L^{\infty}(\R^m))}$ and $f_2\in {L^{p}(\R^n; L^{q}(\R^m))}$ and we can define $T_n \otimes T_m(f_1, f_2) \in L^p(\R^n; L^q(\R^m))$ by setting for all $f_3 \in L^{p'}(\R^n; L^{q'}(\R^m))$
$$
\langle T_n \otimes T_m(f_1,f_2),f_3 \rangle
:= \langle (T_n \otimes T_m)^{1*}(f_3,f_2),f_1 \rangle.
$$
All these readily imply the desired estimate.

In the second case, that is, when $p_2=\infty$ and hence $1<p_1=p<\infty$ we can show much as before that \textbf{Case 1c} yields
\begin{align*}
\|(T_n \otimes T_m)^{2*}(f_1,f_3)\|_{L^{1}(\R^n;L^{q'}(\R^m))}
\lesssim
\|f_1\|_{L^{p}(\R^n; L^{\infty}(\R^m))}
\|f_3\|_{L^{p'}(\R^n; L^{q'}(\R^m))},
\end{align*}
for every $f_1\in {L^{p}(\R^n; L^{\infty}(\R^m))}$  and $f_3 \in L^{p'}(\R^n; L^{q'}(\R^m))$. As a consequence, for any
$f_1\in {L^{p}(\R^n; L^{\infty}(\R^m))}$ and $f_2\in {L^{\infty}(\R^n; L^{q}(\R^m))}$ and we can define $T_n \otimes T_m(f_1, f_2) \in L^p(\R^n; L^q(\R^m))$ by setting for all $f_3 \in L^{p'}(\R^n; L^{q'}(\R^m))$
$$
\langle T_n \otimes T_m(f_1,f_2),f_3 \rangle
:= \langle (T_n \otimes T_m)^{2*}(f_1,f_3),f_2 \rangle.
$$
This gives the desired estimate in the present scenario completing \textbf{Case 1d} and hence \textbf{Case 1}.

\medskip

\noindent\textbf{Case 2:} $1<p_1,p_2\le\infty$ with $\frac1p=\frac1{p_1}+\frac1{p_2}>0$ and $1<q_1,q_2<\infty$.

By Corollary \ref{cor:TensorEndpoint} and a standard density argument we only need to treat the cases where either $p_1=\infty$ or $p_2=\infty$. By symmetry we just explain the  case $p_1 = \infty$, $1<p_2 = p<\infty$. Let $f_1 \in L^{\infty}(\R^n; L^{q_1}(\R^m))$ and $f_2 \in L^{\infty}_c(\R^{n+m})$. Pick $1<\tilde{q}_2<\infty$ large enough so that $\frac1{\tilde{q}}=\frac1{q_1}+\frac1{\tilde{q}_2}<1$. Since $f_2 \in L^{\infty}_c(\R^{n+m})\subset L^{p}(\R^n; L^{\tilde{q}_2}(\R^m))$, our choices of exponents allows us to invoke \textbf{Case 1} we know that $T_n \otimes T_m(f_1, f_2)$ is a well-defined function in $L^p(\R^n; L^{\tilde{q}}(\R^m))$ ---we would like to emphasize that this is the only place in this proof on which we use \textbf{Case 1} and it is done just qualitatively and with the exponents $q_i$'s being finite.
At this point we proceed as in the proof of Corollary \ref{cor:TensorEndpoint} and define $F$, $F_1$ and $F_2$ as in \eqref{afawfwf}. Our goal is to show the validity of \eqref{22qffd}. For $\vec{w}$ as above,
we may assume that $ \|f_1 (w_1\otimes 1)\|_{L^{q_1}(\R^{n+m})}=\| F_1 w_1 \|_{L^{q_1}(\R^n)}$ and $\|f_2 (w_2\otimes 1)\|_{L^{q_2}(\R^{n+m})}=\| F_2 w_2 \|_{L^{q_2}(\R^n)}$ are finite, otherwise there is nothing to prove. That means that we can invoke again Theorem \ref{thm:WeightedTensor}, and \eqref{22qffd} follows in the same manner. We can then extrapolate and in the special case $w_1 \equiv 1$ and $w_2 \equiv 1$ obtain \eqref{eq:eq1afvrv} in the present scenario:
\[
\| T_n \otimes T_m (f_1,f_2) \|_{L^{p}(\R^n; L^{q}(\R^m))}
\lesssim \|f_1 \|_{L^{\infty}(\R^n; L^{q_1}(\R^m))}
\| f_2 \|_{L^{p}(\R^n; L^{q_2}(\R^m))}.
\]
The important fact is that this estimate holds for any $f_1 \in L^{\infty}(\R^n; L^{q_1}(\R^m))$ and $f_2 \in L^{\infty}_c(\R^{n+m})$. In turn, since  $1<p_2,q_2<\infty$ we can run a standard density argument to conclude that it also holds for all $f_1 \in L^{\infty}(\R^n; L^{q_1}(\R^m))$ and $f_2 \in L^{p}(\R^n; L^{q_2}(\R^m))$. This completes the proof. \qed

\end{document}